\documentclass[a4paper, 11pt, american]{amsart}

\usepackage[colorlinks,pdfpagelabels,pdfstartview = FitH,bookmarksopen
= true,bookmarksnumbered = true,linkcolor = blue,plainpages =
false,hypertexnames = false,citecolor = red,pagebackref=false]{hyperref}
\usepackage{times,latexsym,amssymb}
\usepackage{amsmath,amsthm,bm}
\usepackage{graphics, epsfig}
\usepackage{color}
\usepackage{appendix}
\usepackage[margin=1.3in]{geometry}
\usepackage{mathrsfs}

\usepackage{amsbsy}
\usepackage{amstext}
\usepackage{amssymb}
\usepackage{esint}
\usepackage{stmaryrd}

\renewcommand{\d}{\mathrm{d}}
\newcommand{\dx}{\mathrm{d}x}
\newcommand{\dy}{\mathrm{d}y}

\newcommand{\dt}{\mathrm{d}t}
\newcommand{\ds}{\mathrm{d}s}
\renewcommand{\rho}{\varrho}

\let\TeXchi\chi
\newbox\chibox
\setbox0 \hbox{\mathsurround0pt $\TeXchi$}
\setbox\chibox \hbox{\raise\dp0 \box 0 }
\def\chi{\copy\chibox}

\def\Xint#1{\mathchoice
    {\XXint\displaystyle\textstyle{#1}}%
    {\XXint\textstyle\scriptstyle{#1}}%
    {\XXint\scriptstyle\scriptscriptstyle{#1}}%
    {\XXint\scriptscriptstyle\scriptscriptstyle{#1}}%
    \!\int}
\def\XXint#1#2#3{\setbox0=\hbox{$#1{#2#3}{\int}$}
    \vcenter{\hbox{$#2#3$}}\kern-0.5\wd0}
\def\bint{\Xint-}
\def\dashint{\Xint{\raise4pt\hbox to7pt{\hrulefill}}}

\def\XXiint#1#2#3{\setbox0=\hbox{$#1{#2#3}{\iint}$}
    \vcenter{\hbox{$#2#3$}}\kern-0.5\wd0}

\author[N. Liao]{Naian Liao}
\address{Naian Liao,
Fachbereich Mathematik, Universit\"at Salzburg,
Hellbrunner Str. 34, 5020 Salzburg, Austria}
\email{naian.liao@sbg.ac.at}
\newtheorem{proposition}{Proposition}[section]
\newtheorem{theorem}{Theorem}[section]
\newtheorem{lemma}{Lemma}[section]
\newtheorem{corollary}{Corollary}[section]
\newtheorem{remark}{Remark}[section]

\numberwithin{equation}{section}
\numberwithin{theorem}{section}
\numberwithin{proposition}{section}
\numberwithin{lemma}{section}
\numberwithin{remark}{section}
\newcommand{\dsty}{\displaystyle}

\newcommand{\al}{\alpha}

\newcommand{\be}{\beta}

\newcommand{\gm}{\gamma}
\newcommand{\dl}{\delta}

\newcommand{\lm}{\lambda}

\newcommand{\varep}{\varepsilon}

\newcommand{\vp}{\varphi}
\newcommand{\sig}{\sigma}

\newcommand{\om}{\omega}

\newcommand{\z}{\zeta}

\newcommand{\nn}{\mathbb{N}}

\newcommand{\rr}{\mathbb{R}}
\newcommand{\rn}{\rr^N}

\newcommand{\supp}{\operatorname{supp}}
\newcommand{\essup}{\operatornamewithlimits{ess\,sup}}
\newcommand{\essinf}{\operatornamewithlimits{ess\,inf}}
\newcommand{\essosc}{\operatornamewithlimits{ess\,osc}}

\newcommand{\loc}{\operatorname{loc}}

\newcommand{\pl}{\partial}

\newcommand{\intl}{\int\limits}

\def\Xint#1{\mathchoice
    {\XXint\displaystyle\textstyle{#1}}%
    {\XXint\textstyle\scriptstyle{#1}}%
    {\XXint\scriptstyle\scriptscriptstyle{#1}}%
    {\XXint\scriptscriptstyle\scriptscriptstyle{#1}}%
    \!\int}
\def\XXint#1#2#3{\setbox0=\hbox{$#1{#2#3}{\int}$}
    \vcenter{\hbox{$#2#3$}}\kern-0.5\wd0}
\def\bint{\Xint-}
\def\dashint{\Xint{\raise4pt\hbox to7pt{\hrulefill}}}
\def\dashiint{\bint\kern-0.15cm\bint}

\newcommand{\ovl}[3]{\int_{#1}^{#2}\kern-#3pt\raise4pt\hbox to7pt{\hrulefill}\ }

\newcommand{\ovll}[3]{\intl_{#1}^{#2}\kern-#3pt\raise4pt\hbox to7pt{\hrulefill}\ }

\newcommand{\tvl}[2]{\iint_{#1}\kern-#2pt\raise4pt\hbox to7pt{\hrulefill}\ }

\newcommand{\bye}{\end{document}}

\begin{document}
\title[H\"older regularity for parabolic fractional $p$-Laplacian]{H\"older regularity \\ for parabolic fractional $p$-Laplacian}
\maketitle
\begin{abstract}
Local H\"older regularity is established for certain weak solutions to a class of parabolic fractional $p$-Laplace equations with merely measurable kernels. The proof uses DeGiorgi's iteration and refines DiBenedetto's intrinsic scaling method. The control of a nonlocal integral of solutions in the reduction of oscillation plays a crucial role and entails delicate analysis in this intrinsic scaling scenario. Dispensing with any logarithmic estimate and any comparison principle, the proof is new even for the linear case.
\vskip.2truecm
\noindent{\bf Mathematics Subject Classification (2020):} 35R11, 35K65, 35B65, 47G20

\vskip.2truecm
\noindent{\bf Key Words:} H\"older regularity,  parabolic $p$-Laplacian, nonlocal operators, intrinsic scaling
\end{abstract}
\section{Introduction}

We are interested in the local H\"older regularity of weak solutions to
a class of  parabolic equations involving a fractional $p$-Laplacian type operator:
\begin{equation}\label{Eq:1:1}
\pl_t u + \mathscr{L} u=0\quad\text{weakly in}\> E_T,
\end{equation}
where $E_T=E\times(0,T]$ for some open set $E\subset\rn$ and some $T>0$, and the nonlocal operator $\mathscr{L}$ is defined by
\begin{equation}\label{Eq:1:2}
\mathscr{L}u(x,t)={\rm P.V.}\int_{\rn} K(x,y,t)\big|u(x,t) - u(y,t)\big|^{p-2} \big(u(x,t) - u(y,t)\big)\,\dy,
\end{equation}
for some $p>1$.
Here ${\rm P.V.}$ denotes the principle value of the integral, whereas the kernel $K:\rn\times\rn\times(0,T]\to [0,\infty)$ is a measurable function satisfying the following condition uniformly in $t$:
\begin{equation}\label{Eq:K}
\frac{C_o}{|x-y|^{N+sp}}\le K(x,y,t)\equiv K(y,x,t)\le \frac{C_1}{|x-y|^{N+sp}}\quad\text{a.e.}\> x,\,y\in\rn,
\end{equation}
for some positive $C_o$, $C_1$ and $s\in(0,1)$. 

Throughout this note,  the parameters $\{s, p, N, C_o, C_1\}$ are termed the {\it data} and we use $\boldsymbol\gm$ as a generic positive constant in various estimates that can be determined by the data only.

The formal definition of weak solution to \eqref{Eq:1:1} -- \eqref{Eq:K} and notations can be found in Section~\ref{S:notion}. 
We proceed to present our main result as follows.

\begin{theorem}\label{Thm:1:1}
Let $u$ be a locally bounded, local, weak solution to \eqref{Eq:1:1} -- \eqref{Eq:K} in $E_T$ with $p>1$.
	Then $u$ is locally H\"older continuous in $E_T$. More precisely,
	there exist constants $\boldsymbol\gm>1$ and $\be\in(0,1)$ that can be determined a priori
	only in terms of the data, such that for any $0<\rho<R<\widetilde{R}$, there holds
	\begin{equation*}
	\essosc_{(x_o,t_o)+Q_{\rho}(\boldsymbol\om^{2-p})} u\le \boldsymbol\gm \boldsymbol\om \Big(\frac{\rho}{R}\Big)^{\be},
	\end{equation*}
provided the  cylinders  $(x_o,t_o)+Q_{R}(\boldsymbol\om^{2-p})\subset (x_o,t_o)+Q_{\widetilde{R}}$  
are included in $E_T$, 
where
$$\boldsymbol\om=2\essup_{(x_o,t_o)+Q_{\widetilde{R}}}|u| +{\rm Tail}\big(u; (x_o,t_o)+Q_{\widetilde{R}}\big).$$
\end{theorem}

\begin{remark}\upshape
  Theorem~\ref{Thm:1:1} has been formulated independent of any initial/boundary data. While local,  the oscillation estimate bears global information via the {\it tail} of $u$; see~\eqref{Eq:tail}. In particular, a solution  is allowed to grow at infinity. Whereas if $u$ is globally bounded in $\rn\times(0,T)$, then   $\boldsymbol\om$ can be taken as the global bound. This occurs if, for instance, proper initial/boundary data are prescribed, cf.~\cite{BLS, Caff-Vass-11}. In addition, if $u$ is globally bounded in $\rn\times(-\infty,T)$, then $u$ is,  {\it a fortiori}, a constant by the oscillation estimate.
\end{remark}
\begin{remark}\upshape
  Theorem~\ref{Thm:1:1} continues to hold for more general structures. For instance, one can consider the kernel satisfying
 \[
  \frac{C_o \chi_{\{|x-y|\le 1\}}}{|x-y|^{N+sp}}\le K(x,y,t)\equiv K(y,x,t)\le \frac{C_1}{|x-y|^{N+sp}}\quad\text{a.e.}\> x,\,y\in\rn.
 \]
 Also, proper lower order terms can be considered, cf.~\cite{Cozzi}.
However, we will concentrate on the actual novelty and  leave possible generalizations to the motivated reader.
\end{remark}
\subsection{Novelty and Significance}
The nonlocal elliptic operator $\mathscr{L}$ as in \eqref{Eq:1:2} with a kernel like \eqref{Eq:K}, especially when $p=2$, has been a classical topic in Probability, Potential theory, Harmonic Analysis, etc. In addition, nonlocal partial differential equations arise from continuum mechanics and phase transition, from population dynamics, and from optimal control and game theory. We refer to \cite{Caff-12, Hitchhiker} for a source of motivations and applications.

Local regularity for the nonlocal elliptic operator with merely measurable kernels is well studied, cf.~ \cite{Cozzi, DKP-1, DKP-2, Kass-09, Silv-06}  -- just to mention a few.  In \cite{DKP-1, DKP-2}, localization techniques are developed  in order to establish H\"older regularity and Harnack's inequality for the elliptic operator. A logarithmic estimate plays a key role in \cite{DKP-1, DKP-2}. Whereas \cite{Cozzi} further improves these results to functions in certain DeGiorgi classes and the logarithmic estimate is dispensed with. 

The parabolic nonlocal problem \eqref{Eq:1:1} has witnessed a growing interest recently; 
see \cite{BGK-1, BGK, BLS, Caff-Vass-11, DZZ, Kass-13, Kim-0, Kim-1, Mazon-16, Stromqvist-1, Stromqvist-2, Vaz-16, Vaz-20} -- just to mention a few.
Coming to the local regularity, while the case $p=2$ has been subject to extensive studies, 
the case $p\neq2$ is largely open. Local regularity of supersolutions is studied in \cite{BGK-1, BGK}. Some local boundedness estimates are reported in \cite{DZZ, Stromqvist-1}. Meanwhile, it is tried in \cite{DZZ} to adapt techniques of \cite{DKP-1} and to show H\"older regularity for \eqref{Eq:1:1} with $p>2$. However, we are unable to verify (5.23), which is based on a logarithmic estimate. 
When the kernel $K(x,y,t)$ is exactly $ 2|x-y|^{-N-sp}$ for $p\ge2$, explicit H\"older exponents are obtained in \cite{BLS}.

Our contribution lies in establishing H\"older regularity for the parabolic fractional $p$-Laplace type equation with merely measurable kernels for all $p>1$. The approach is structural, in the sense that we dispense with any kind of comparison principle. More generally, we find that the H\"older regularity is in fact encoded in a family of energy estimates in Corollary~\ref{Cor:1} and tools like logarithmic estimates or exponential change of variables play no role in our arguments.\footnote{After completion of the paper, we noticed a similar result in the recent preprint \cite{Adi}, via exponential change of variables. } As such, the arguments are new even for $p=2$ and they hold the promise of a wider applicability, for instance, in Calculus of Variations. 

Unlike the elliptic operator or the parabolic operator with $p=2$, the local behavior of a solution to the parabolic $p$-Laplacian is markedly different: it has to be read in its own {\it intrinsic geometry}. This is the guiding idea in the local operator theory, cf.~\cite{ChenDB92, DB, DB86, DBGV-mono, Urbano-08}. In terms of oscillation estimates, this idea leads to the construction of geometric sequences $\{R_i\}$ and $\{\boldsymbol\om_{i}\}$ connected by the intrinsic relation
\begin{equation}\label{Eq:osc-intro}
\begin{array}{c}
 \dsty 
 \essosc_{Q_{R_{i}}(\boldsymbol\om_{i}^{2-p})}u\le\boldsymbol\om_{i}.
\end{array}
\end{equation}
The nonlocal theory developed here is no exception. However, the nonlocal character of \eqref{Eq:1:1} needs to be carefully handled in this intrinsic scaling scenario. A new component brought by the nonlocality of the operator is a proper control of the so-called {\it tail} -- a nonlocal integral of the solution (see~\eqref{Eq:tail}). Precisely, we have
\begin{equation}\label{Eq:tail-intro}
{\rm Tail}\big(\big(u - \boldsymbol \mu_i^{\pm}\big)_{\pm}; Q_{R_{i}}(\boldsymbol\om_{i}^{2-p})\big)  \le \boldsymbol\gm\boldsymbol\om_i,
\end{equation}
where $ \boldsymbol \mu_i^{\pm}$ denotes the supremum/infimum of $u$ over $Q_{R_{i}}(\boldsymbol\om_{i}^{2-p})$. In other words, the nonlocal tail is controlled by the local oscillation, if the intrinsic relation \eqref{Eq:osc-intro} is verified. The tail estimate \eqref{Eq:tail-intro} in turn allows us to reduce the oscillation in the next step, and so on. This induction procedure can be illustrated by 
$$\eqref{Eq:osc-intro}_{i}^{i+1}\> \leftrightarrows \> \eqref{Eq:tail-intro}_{i}.$$

The local regularity theory for the nonlocal parabolic problem \eqref{Eq:1:1} with $p\neq2$ is still at its inception. We believe the techniques developed in this note are flexible enough and provide a handy toolkit that can be used to fruitfully attack more general nonlocal parabolic equations. 
\subsection{Definitions and Notations}\label{S:notion}
\subsubsection{Function Spaces}
For $p>1$ and $s\in(0,1)$, we introduce the fractional Sobolev space $W^{s,p}(\rn)$ by
\[
W^{s,p}(\rn):=\bigg\{v \in L^p(\rn):\, \int_{\rn}\int_{\rn}\frac{|v(x) - v(y)|^p}{|x-y|^{N+sp}}\,\dx\dy<\infty\bigg\},
\]
which is endowed with the norm
\[
\|v\|_{W^{s,p}(\rn)}:=\bigg(\int_{\rn}|v|^p\,\dx\bigg)^{\frac1p} + \bigg(\int_{\rn}\int_{\rn}\frac{|v(x)-v(y)|^p}{|x-y|^{N+sp}}\,\dx\dy\bigg)^{\frac1p}.
\]
Similarly, the fractional Sobolev space $W^{s,p}(E)$ for a domain $E\subset\rn$ can be defined. Moreover, we denote
\[
W^{s,p}_o(E):=\bigg\{v\in W^{s,p}(\rn):\, v=0\>\>\text{a.e. in}\>\rn\setminus E \bigg\}.
\]
These spaces admit imbedding into proper Lebesgue spaces; we collect some in Appendix~\ref{A:1}.

\subsubsection{Notion of Weak Solution}
A function
\begin{equation*} 
	u\in C_{\loc}\big(0,T;L^2_{\loc}(E)\big)\cap L^p_{\loc}\big(0,T; W^{s,p}_{\loc}(E)\big)
\end{equation*}
is a local, weak sub(super)-solution to \eqref{Eq:1:1} -- \eqref{Eq:K}, if for every compact set $\mathcal{K}\subset E$ and every sub-interval
$[t_1,t_2]\subset (0,T]$, we have
\begin{equation}\label{Eq:global-int}
\essup_{t_1<t<t_2}\int_{\rn}\frac{|u(x,t)|^{p-1}}{1+|x|^{N+sp}}\,\dx<\infty
\end{equation}
and
\begin{equation}  \label{Eq:1:4p}
\begin{aligned}
	\int_{\mathcal{K}} & u\vp \,\dx\bigg|_{t_1}^{t_2}
	-\int_{t_1}^{t_2}\int_{\mathcal{K}}  u\pl_t\vp\, \dx\dt
	+\int_{t_1}^{t_2} \mathscr{E}\big(u(\cdot, t), \vp(\cdot, t)\big)\,\dt
	\le(\ge)0
\end{aligned}
\end{equation}
where
\[
	 \mathscr{E}:=\int_{\rn}\int_{\rn}K(x,y,t)\big|u(x,t) - u(y,t)\big|^{p-2} \big(u(x,t) - u(y,t)\big)\big(\vp(x,t) - \vp(y,t)\big)\,\dy\dx
\]
for all non-negative testing functions
\begin{equation}\label{Eq:test-function}
\vp \in W^{1,2}_{\loc}\big(0,T;L^2(\mathcal{K})\big)\cap L^p_{\loc}\big(0,T;W_o^{s,p}(\mathcal{K})%
\big).
\end{equation}

A function $u$ that is both a local weak sub-solution and a local weak super-solution
to \eqref{Eq:1:1} -- \eqref{Eq:K} is a local weak solution.

\begin{remark}\upshape
As we are developing a local theory, the function space in \eqref{Eq:test-function} can be taken smaller. Namely, the notion of $W_o^{s,p}(\mathcal{K})$ can be replaced by functions $\vp(\cdot, t)\in W^{s,p}(\rn)$ with a compact support in $\mathcal{K}$ for a.e. $t$. 
\end{remark}

\begin{remark}\upshape
To ensure the convergence of the global integral in \eqref{Eq:1:4p}, it suffices to weaken the $L^\infty$ norm appearing in the condition \eqref{Eq:global-int} by the $L^1$ norm. However,  in deriving the energy estimate of Proposition~\ref{Prop:2:1}, the condition \eqref{Eq:global-int} is needed already.
\end{remark}

\subsubsection{Some Notations}
Throughout this note, 
we will use $K_\rho(x_o)$ to denote the ball of radius $\rho$ and center $x_o$ in $\rn$, and the symbols 
\begin{equation*}
\left\{
\begin{aligned}
(x_o,t_o)+Q_\rho(\theta)&:=K_{\rho}(x_o)\times(t_o-\theta\rho^{sp},t_o),\\[5pt]
(x_o,t_o)+Q(R,S)&:=K_R(x_o)\times (t_o-S,t_o),
\end{aligned}\right.
\end{equation*} 
to denote (backward) cylinders with the indicated positive parameters.
When the context is unambiguous, we will omit the vertex $(x_o,t_o)$ from the symbols for simplicity.
When $\theta=1$, it is also omitted.

A nonlocal integral of $u$ -- termed the {\it tail } of $u$ -- inevitably appears in the theory, which we define as
\begin{equation}\label{Eq:tail}
{\rm Tail}\big(u; Q(R,S)\big):= \essup_{t_o-S<t<t_o}\bigg(R^{sp}\int_{\rn\setminus K_R(x_o)}\frac{|u(x,t)|^{p-1}}{|x-x_o|^{N+sp}}\,\dx \bigg)^{\frac1{p-1}}.
\end{equation}
For any $Q(R,S)\subset E_T$, the finiteness of this tail is guaranteed by \eqref{Eq:global-int}.


\section{Energy Estimates}\label{S:energy}
This section is devoted to energy estimates satisfied by local weak sub(super)-solutions to \eqref{Eq:1:1} -- \eqref{Eq:K}.
We first introduce,
 for any $k\in\rr$, the truncated functions
\[
(u-k)_+=\max\{u-k,0\}, \qquad (u-k)_-=\max\{-(u-k),0\}.
\]

In what follows, when we state {\it ``$u$ is a sub(super)-solution...''}
and use $``\pm"$ or $``\mp"$ in what follows, we mean the sub-solution corresponds to
the upper sign and the super-solution corresponds to the lower sign in the statement.

\begin{proposition}\label{Prop:2:1}
	Let $u$ be a  local weak sub(super)-solution to \eqref{Eq:1:1} -- \eqref{Eq:K} in $E_T$.
	There exists a constant $\boldsymbol \gm (C_o,C_1,p)>0$, such that
 	for all cylinders $Q(R,S) \subset E_T$,
 	every $k\in\rr$, and every non-negative, piecewise smooth cutoff function
 	$\z(\cdot,t)$ compactly supported in $ K_{R} $ for all  $t\in (t_o-S,t_o)$,  there holds
\begin{align*}
	&\int_{t_o-S}^{t_o}\int_{K_R}\int_{K_R}\min\big\{\z^p(x,t),\z^p(y,t)\big\} \frac{|w_\pm(x,t) - w_\pm(y,t)|^p}{|x-y|^{N+sp}}\,\dx\dy\dt\\
	&\qquad+\iint_{Q(R,S)} \z^p w_{\pm}(x,t)\,\dx\dt \bigg(\int_{ K_R}  \frac{w^{p-1}_\mp(y,t)}{|x-y|^{N+sp}}\,\dy\bigg)
	 +\int_{K_R}\z^p w^2_{\pm}(x,t)\,\dx\bigg|_{t_o-S}^{t_o}\\
	&\quad\le
	\boldsymbol\gm\int_{t_o-S}^{t_o}\int_{K_R}\int_{K_R}\max\big\{w^p_{\pm}(x,t), w^p_{\pm}(y,t)\big\} \frac{|\z(x,t) - \z(y,t)|^p}{|x-y|^{N+sp}}\,\dx\dy\dt\\
	&\qquad+\boldsymbol\gm\iint_{Q(R,S)} \z^pw_{\pm}(x,t)\,\dx\dt \bigg(\essup_{\substack{x\in\supp\z(\cdot, t)\\t\in(t_o-S,t_o)}} \int_{\rn\setminus K_R}\frac{ w_{\pm}^{p-1}(y,t)}{|x-y|^{N+sp}}\,\dy\bigg)\\
	&\qquad+ \iint_{Q(R,S)} |\pl_t\z^p|w_{\pm}^2(x,t)\,\dx\dt.  
\end{align*}
Here, we have denoted $w=u-k$ for simplicity.
\end{proposition}
\begin{proof}
We will only deal with the case of sub-solution as the other case is similar. Using $\vp=w_+\z^p$ as a testing function in the weak formulation modulo a proper time mollification  (cf.~Appendix~\ref{A:2}), the last terms on the right/left-hand side of the energy estimate are rather standard. We only treat the integral resulting from the fractional diffusion part, which, due to the support of $\z$ and symmetry of the integrand, can be split into two parts, that is,
\begin{align*}
&\int_{t_o-S}^{t_o}\int_{\rn}\int_{\rn} \big|u(x,t) - u(y,t)\big|^{p-2}\big(u(x,t) - u(y,t)\big)\big(\vp(x,t) - \vp(y,t)\big)\,\d\mu\\
&\quad=\int_{t_o-S}^{t_o}\int_{K_R}\int_{K_R} \big|u(x,t) - u(y,t)\big|^{p-2}\big(u(x,t) - u(y,t)\big)\big(\vp(x,t) - \vp(y,t)\big)\,\d\mu\\
&\qquad+2\int_{t_o-S}^{t_o}\int_{K_{R}}\int_{\rn\setminus K_R} \big|u(x,t) - u(y,t)\big|^{p-2}\big(u(x,t) - u(y,t)\big) \vp(x,t)  \,\d\mu\\
&\quad=:I_1+I_2,
\end{align*}
where we have used $\d\mu=K(x,y,t)\dy\dx\dt$ for simplicity. 

Let us manipulate the first integral, which is the leading term. 
To this end, we denote $A_k=\big\{u(\cdot, t)>k\big\}\cap K_R$ for fixed $t\in(t_o-S,t_o)$.
Observe that if $x\in A_k$ while $y\in K_R\setminus A_k$, one obtains
\begin{equation}\label{Eq:aux:1}
\begin{aligned}
\big|u(x,t) &- u(y,t)\big|^{p-2}\big(u(x,t) - u(y,t)\big)\big(\vp(x,t) - \vp(y,t)\big)\\
&=
\big(w_+(x,t)+w_-(y,t)\big)^{p-1} w_+(x,t)\z^p(x,t)\\
&\ge c(p) w_+^p(x,t) \z^p(x,t)+c(p) w^{p-1} _-(y,t)w_+(x,t)\z^p(x,t),
\end{aligned}
\end{equation}
for some proper $c=c(p)$.
Whereas if $x,y\in A_k$, we claim that
\begin{equation}\label{Eq:aux:2}
\begin{aligned}
\big|u(x,t) &- u(y,t)\big|^{p-2}\big(u(x,t) - u(y,t)\big)\big(\vp(x,t) - \vp(y,t)\big)\\
&\ge
\tfrac12\big|w_+(x,t)-w_+(y,t)\big|^{p} \max\big\{\z^p(x,t), \z^p(y,t)\big\}\\
&\quad - \boldsymbol\gm(p)\max\big\{w_+^p(x,t), w_+^p(y,t)\big\}\big|\z(x,t)-\z(y,t)\big|^p.
\end{aligned}
\end{equation}

To prove the claim, we may assume that $u(x,t)\ge u(y,t)$ due to symmetry and write
\begin{equation}\label{Eq:aux:3}
\begin{aligned}
\big|u(x,t) &- u(y,t)\big|^{p-2}\big(u(x,t) - u(y,t)\big)\big(\vp(x,t) - \vp(y,t)\big)\\
&=
\big(w_+(x,t)-w_+(y,t)\big)^{p-1} \big(w_+(x,t)\z^p(x,t) - w_+(y,t)\z^p(y,t)\big).
\end{aligned}
\end{equation}
If $\z(x,t)\ge\z(y,t)$, the above display \eqref{Eq:aux:3} is estimated below by
\[
\big(w_+(x,t)-w_+(y,t)\big)^{p}  \z^p(x,t),
\]
and hence the claim follows. If, instead $\z(x,t)<\z(y,t)$, then \eqref{Eq:aux:3} can be written as
\[
\big(w_+(x,t)-w_+(y,t)\big)^{p}  \z^p(y,t)-\big(w_+(x,t)-w_+(y,t)\big)^{p-1} w_+(x,t) \big(\z^p(y,t) - \z^p(x,t)\big).
\]
To proceed, we need an elementary inequality:
\begin{equation}\label{Eq:algebra}
a^p-b^p\le\varep a^p + \frac{p^p}{\varep^{p-1}}(a-b)^p\quad\text{for}\> a\ge b\ge0,\>\varep>0.
\end{equation}
This simply follows from the mean value theorem and Young's inequality:
\[
a^p-b^p\le p a^{p-1}(a-b)\le \varep a^p + \frac{p^p}{\varep^{p-1}}(a-b)^p.
\]
We may apply \eqref{Eq:algebra} with $a=\z(y,t)$, $b=\z(x,t)$ and $$\varep=\big(w_+(x,t)-w_+(y,t)\big)/2w_+(x,t)$$ to obtain
\begin{align*}
\big(w_+(x,t)&-w_+(y,t)\big)^{p-1} w_+(x,t) \big(\z^p(y,t) - \z^p(x,t)\big)\\
&\le\tfrac12\big(w_+(x,t)-w_+(y,t)\big)^{p}  \z^p(y,t)+\boldsymbol\gm(p)w_+^p(x,t) \big(\z(y,t) - \z(x,t)\big)^p.
\end{align*}
Combining the last two estimates in \eqref{Eq:aux:3}, we obtain \eqref{Eq:aux:2} in the case $\z(x,t)<\z(y,t)$ also.

Employing \eqref{Eq:aux:1} and \eqref{Eq:aux:2} and properly adjusting $c$ if necessary, the first integral $I_1$ is estimated by
\begin{align*}
I_1&\ge
c\int_{t_o-S}^{t_o}\int_{A_k}\int_{A_k} \big|w_+(x,t)-w_+(y,t)\big|^{p} \max\big\{\z^p(x,t), \z^p(y,t)\big\}\,\d\mu\\
&\quad+2c\int_{t_o-S}^{t_o}\int_{A_{k}}\int_{K_R\setminus A_k} \big|w_+(x,t)-w_+(y,t)\big|^{p}  \z^p(x,t) \,\d\mu\\
&\quad+2c\int_{t_o-S}^{t_o}\int_{A_{k}}\int_{K_R\setminus A_k} w_-^{p-1}(y,t)  w_+(x,t)\z^p(x,t) \,\d\mu\\
&\quad - \boldsymbol\gm(p)\int_{t_o-S}^{t_o}\int_{K_R}\int_{K_R}\max\big\{w_+^p(x,t), w_+^p(y,t)\big\}\big|\z(x,t)-\z(y,t)\big|^p\,\d\mu\\
&\ge c\int_{t_o-S}^{t_o}\int_{K_R}\int_{K_R} \big|w_+(x,t)-w_+(y,t)\big|^{p} \min\big\{\z^p(x,t), \z^p(y,t)\big\}\,\d\mu\\
&\quad+c\int_{t_o-S}^{t_o}\int_{K_R}\int_{K_R} w_-^{p-1}(y,t)  w_+(x,t)\z^p(x,t) \,\d\mu\\
&\quad - \boldsymbol\gm(p)\int_{t_o-S}^{t_o}\int_{K_R}\int_{K_R}\max\big\{w_+^p(x,t), w_+^p(y,t)\big\}\big|\z(x,t)-\z(y,t)\big|^p\,\d\mu.
\end{align*}
 
 Now let us treat the second integral $I_2$, which yields the only nonlocal integral in the energy estimate. Indeed, we first estimate
 \begin{align*}
-\big|u(x,t)& - u(y,t)\big|^{p-2}\big(u(x,t) - u(y,t)\big) \big(u(x,t)-k\big)_+\\
&\le \big(u(y,t)-u(x,t)\big)^{p-1}_+\big(u(x,t)-k\big)_+\\
&\le \big(u(y,t)-k\big)^{p-1}_+\big(u(x,t)-k\big)_+.
 \end{align*}
 As a result, 
 we may estimate by the condition \eqref{Eq:K} on the kernel  $K$,
\begin{align*}
-I_2&\le2\int_{t_o-S}^{t_o}\int_{K_{R}}\int_{\rn\setminus K_R} w_+^{p-1}(y,t) w_+(x,t)\z^p(x,t)\,\d\mu\\
&\le \boldsymbol\gm \iint_{Q(R,S)} \z^pw_{+}(x,t)\,\dx\dt \bigg(\essup_{\substack{x\in\supp\z(\cdot, t)\\t\in(t_o-S,t_o)}} \int_{\rn\setminus K_R}\frac{ w_{+}^{p-1}(y,t)}{|x-y|^{N+sp}}\,\dy\bigg).
\end{align*}
Note that the finiteness of the above nonlocal integral is guaranteed by \eqref{Eq:global-int}. This term will evolve into the tail term \eqref{Eq:tail} in the forthcoming theory.

Finally, we can put all these estimates together and use the condition \eqref{Eq:K} on the kernel  $K$ to conclude.
\end{proof}

The above energy estimate can be written in $K_R\times(t_o- S, t)$ for any $t\in(t_o- S, t_o)$. As usual, this will lead to an $L^\infty$ estimate in the time variable on the left, due to the arbitrariness of $t$. Further, by choosing a proper cutoff function $\z$, we derive the following two types of energy estimates from Proposition~\ref{Prop:2:1}, which encode all the information needed to show Theorem~\ref{Thm:1:1}.
\begin{corollary}\label{Cor:1}
	Let $u$ be a  local weak sub(super)-solution to \eqref{Eq:1:1} -- \eqref{Eq:K} in $E_T$.
	There exists a constant $\boldsymbol \gm (C_o,C_1,p)>0$, such that
 	for all cylinders $Q(r,\tau)\subset Q(R,S) \subset E_T$,
 	and every $k\in\rr$, 
	there holds
\begin{align*}
\essup_{t_o-\tau<t<t_o}\int_{K_r} &w^2_{\pm}(x,t)\,\dx+
	\int_{t_o-\tau}^{t_o}\int_{K_r}\int_{K_r}  \frac{|w_\pm(x,t) - w_\pm(y,t)|^p}{|x-y|^{N+sp}}\,\dx\dy\dt\\
	&\le
	\frac{\boldsymbol\gm R^{(1-s)p}}{(R-r)^p}\int_{t_o-S}^{t_o}\int_{K_R}  w^p_{\pm}(x,t) \,\dx\dt\\
	&\quad+\frac{\boldsymbol\gm R^{N}}{(R-r)^{N+sp}}\iint_{Q(R,S)} w_{\pm}(x,t)\,\dx\dt 
	\big[{\rm Tail}\big(w_\pm; Q(R,S)\big)\big]^{p-1}\\
	&\quad+ \frac{\boldsymbol\gm}{S-\tau}\iint_{Q(R,S)} w_{\pm}^2(x,t)\,\dx\dt  
\end{align*}
and 
\begin{align*}
	 \essup_{t_o-S<t<t_o}\int_{K_r}& w^2_{\pm}(x,t)\,\dx + \int_{t_o-S}^{t_o}\int_{K_r}\int_{K_r}  \frac{|w_\pm(x,t) - w_\pm(y,t)|^p}{|x-y|^{N+sp}}\,\dx\dy\dt\\
	 &\quad+\frac1{r^{N+sp}}\int_{t_o-S}^{t_o}\int_{K_r}\int_{K_r}   w_{\pm}(x,t) w^{p-1}_\mp(y,t)\,\dy\dx\dt  \\
	 &\le \int_{K_R} w^2_\pm(x,t_o-S)\,\dx
	+\frac{\boldsymbol\gm R^{(1-s)p}}{(R-r)^p}\int_{t_o-S}^{t_o}\int_{K_R}  w^p_{\pm}(x,t) \,\dx\dt\\
	&\quad +\frac{\boldsymbol\gm R^{N}}{(R-r)^{N+sp}}\iint_{Q(R,S)} w_{\pm}(x,t)\,\dx\dt
	\big[{\rm Tail}\big(w_\pm; Q(R,S)\big)\big]^{p-1}. 
\end{align*}
Here, we have denoted $w=u-k$ for simplicity.
\end{corollary}

\section{Preliminary Tools} \label{S:prelim}


In this section, we collect the main modules of the proof of H\"older regularity. 
An important feature is that the tail term appears in these modules via an either-or form. This feature 
clarifies the role of the tail and greatly facilitates the delicate intrinsic scaling arguments to be unfolded in the next two sections. To streamline, we derive them from the energy estimates in Proposition~\ref{Prop:2:1}. Nevertheless, it will be clear from their proofs that Corollary~\ref{Cor:1} actually suffices. We also stress that the arguments in this section are given in a unified fashion for all $p>1$.

Throughout this section, let   
 $\mathcal{Q}:=K_R(x_o)\times(T_1,T_2]$ be a cylinder included in $E_T$.
We introduce numbers $\boldsymbol\mu^{\pm}$ and $\boldsymbol\om$ satisfying
\begin{equation*}
	\boldsymbol\mu^+\ge\essup_{\mathcal{Q}}u,
	\quad 
	\boldsymbol\mu^-\le\essinf_{\mathcal{Q}} u,
	\quad
	\boldsymbol\om\ge\boldsymbol\mu^+-\boldsymbol\mu^-.
\end{equation*}

The first result concerns a DeGiorgi type lemma. 
For simplicity, we omit the vertex $(x_o,t_o)$ from $Q_\rho(\theta)$. 
\begin{lemma}\label{Lm:DG:1}
 Let $u$ be a locally bounded, local weak sub(super)-solution to \eqref{Eq:1:1} -- \eqref{Eq:K} in $E_T$.
 For some $ \dl,\,\xi\in(0,1)$ set $\theta=\dl(\xi\boldsymbol\om)^{2-p}$ and assume $ Q_\varrho(\theta) \subset \mathcal{Q}$.
There exists a constant $\nu\in(0,1)$ depending only on 
 the data  and $\dl$, such that if
\begin{equation*}
	\Big|\Big\{
	\pm\big(\boldsymbol \mu^{\pm}-u\big)\le \xi\boldsymbol\om\Big\}\cap  Q_{\varrho}(\theta)\Big|
	\le
	\nu|Q_{\varrho}(\theta)|,
\end{equation*}
then either
\begin{equation*}
	 \Big( \frac{\rho}{R}\Big)^{\frac{sp}{p-1}}  {\rm Tail}\big( \big(u - \boldsymbol \mu^{\pm}\big)_{\pm}; \mathcal{Q}\big)>\xi\boldsymbol\om,
\end{equation*}
or
\begin{equation*}
	\pm\big(\boldsymbol\mu^{\pm}-u\big)\ge \tfrac12\xi\boldsymbol\om
	\quad
	\mbox{a.e.~in $ Q_{\frac{1}2\varrho}(\theta)$.}
\end{equation*}
Moreover, we have the dependence $\nu\approx  \dl^q$ for some $ q>1$ depending on $p$ and $N$.
\end{lemma}

\begin{proof}
 It suffices to show the case of super-solution with $\boldsymbol\mu^-=0$. Assume $(x_o,t_o)=(0,0)$ and
define for $n\in \nn\cup\{0\}$,
\begin{align*}
	\left\{
	\begin{array}{c}
\dsty k_n= \frac{\xi\boldsymbol\om}{2} +\frac{\xi\boldsymbol\om }{2^{n+1}},
\\[5pt]
\dsty \rho_n=\frac{\rho}2+\frac{\rho}{2^{n+1}},\quad\tilde{\rho}_n=\frac{\rho_n+\rho_{n+1}}2,\\[5pt]
\dsty \hat{\rho}_n=\frac{3\rho_n+ \rho_{n+1}}{4},\quad \bar{\rho}_n=\frac{\rho_n+ 3\rho_{n+1}}{4},\\[5pt]
\dsty K_n=K_{\rho_n}, \quad \widetilde{K}_n=K_{\tilde{\rho}_n},\quad \widehat{K}_n=K_{\hat\rho_n},\quad \overline{K}_n=K_{\bar{\rho}_n},\\[5pt]
\dsty Q_n=K_n\times\big(-\theta\rho_n^{sp},0\big),\quad\widetilde{Q}_n=\widetilde{K}_n\times\big(-\theta\tilde{\rho}_n^{sp},0\big),\\[5pt]
\widehat{Q}_n=\widehat{K}_n\times\big(-\theta\hat{\rho}_n^{sp},0\big),\quad\overline{Q}_n=\overline{K}_n\times\big(-\theta\bar{\rho}_n^{sp},0\big).
\end{array}
	\right.
\end{align*}
It is helpful to observe that $$Q_{n+1}\subset\overline{Q}_n\subset\widetilde{Q}_n\subset\widehat{Q}_n\subset Q_n.$$
Introduce the cutoff function $\z$ in $Q_n$, vanishing outside $\widehat{Q}_{n}$ and
equal to the identity in $\widetilde{Q}_{n}$, such that
\begin{equation*}
|D\z|\le \frac{2^n}{\rho}\quad\text{ and }\quad |\pl_t \z|\le \frac{2^{psn}}{\theta\rho^{sp}}.
\end{equation*}
Let us examine the energy estimate of Proposition~\ref{Prop:2:1} 
in this setting:
\begin{align*}
	\essup_{-\theta\tilde\rho_n^{sp}<t<0}&\int_{\widetilde{K}_n} w^2_{-}(x,t)\,\dx+ \int_{-\theta\tilde\rho_n^p}^{0}\int_{\widetilde{K}_n}\int_{\widetilde{K}_n}  \frac{|w_-(x,t) - w_-(y,t)|^p}{|x-y|^{N+sp}}\,\dx\dy\dt\\
	&\quad\le
	\boldsymbol\gm\int_{-\theta\rho_n^{sp}}^{0}\int_{K_n}\int_{K_n}\max\big\{w^p_{-}(x,t), w^p_{-}(y,t)\big\} \frac{|\z(x,t) - \z(y,t)|^p}{|x-y|^{N+sp}}\,\dx\dy\dt\\
	&\qquad+\boldsymbol\gm\iint_{Q_n} \z^p w_{-}(x,t)\,\dx\dt \bigg(\essup_{\substack{x\in \widehat{K}_n\\t\in(-\theta\rho_n^{sp},0)}} \int_{\rn\setminus K_n}\frac{ w_{-}^{p-1}(y,t)}{|x-y|^{N+sp}}\,\dy\bigg)\\
	&\qquad+ \iint_{Q_n} |\pl_t\z^p|w_{-}^2(x,t)\,\dx\dt.  
\end{align*}

Recalling $w_-=(u-k_n)_-$, we treat the three terms on the right-hand side of the energy estimate  as follows.
For the first term, we estimate
\begin{align*}
\int_{-\theta\rho_n^{sp}}^{0}&\int_{K_n}\int_{K_n}\max\big\{w_{-}(x,t), w_{-}(y,t)\big\}^p \frac{|\z(x,t) - \z(y,t)|^p}{|x-y|^{N+sp}}\,\dx\dy\dt\\
&\le 2^{pn+1} \frac{(\xi\boldsymbol\om)^p}{\rho^p}\int_{-\theta\rho_n^p}^{0}\int_{K_n}\int_{K_n} \frac{ \chi_{\{u(x,t)<k_n\}} }{|x-y|^{N+(s-1)p}}\,\dx\dy\dt\\
&\le \boldsymbol \gm 2^{pn} \frac{(\xi\boldsymbol\om)^p}{\rho^{sp}} |A_n|,
\end{align*}
where we have defined $A_n:=\{u<k_n\}\cap Q_n$.
 
For the second term, we observe that for $|y|\ge \rho_n$ and $|x|\le \hat\rho_n$, there holds
\[
\frac{|y-x|}{|y|}\ge1-\frac{\hat\rho_n}{ \rho_n}=\frac14\Big(\frac{\rho_n-\rho_{n+1}}{\rho_n}\Big)\ge\frac1{2^{n+4}};
\]
consequently, recalling also $u\ge\boldsymbol \mu^-=0$ a.e. in $\mathcal{Q}$ by assumption, we estimate
\begin{align*}
\iint_{Q_n} & \z^pw_{-}(x,t)\,\dx\dt \bigg(\sup_{\substack{x\in\widehat{K}_n\\t\in(-\theta\rho_n^{sp},0)}} \int_{\rn\setminus K_n}\frac{ w_{-}^{p-1}(y,t)}{|x-y|^{N+sp}}\,\dy\bigg)\\
&\le \boldsymbol\gm 2^{(N+sp)n} \xi\boldsymbol\om |A_n| \bigg(\boldsymbol\gm 2^{spn} \frac{(\xi\boldsymbol\om)^{p-1}}{\rho^{sp}}+\essup_{\substack{t\in(-\theta\rho_n^{sp},0)}} \int_{\rn\setminus K_R}\frac{ u_{-}^{p-1}(y,t)}{|y|^{N+sp}}\,\dy\bigg)\\
&= \boldsymbol\gm 2^{(N+sp)n} \frac{\xi\boldsymbol\om}{\rho^{sp}} |A_n| \bigg(\gm 2^{spn} (\xi\boldsymbol\om)^{p-1} + \Big( \frac{\rho}{R}\Big)^{sp} [{\rm Tail}(u_-; \mathcal{Q})]^{p-1} \bigg)\\
&\le \boldsymbol\gm 2^{(N+2sp)n} \frac{(\xi\boldsymbol\om)^p}{\rho^{sp}} |A_n|.
\end{align*}
  In the last line, we have enforced
$$ \Big( \frac{\rho}{R}\Big)^{\frac{sp}{p-1}}   {\rm Tail}(u_{-}; \mathcal{Q})  \le \xi\boldsymbol\om.$$

For the third term, it is quite standard to obtain
\[
\iint_{Q_n} |\pl_t\z^p |(u-k_n)_{-}^2\,\dx\dt\le\frac{2^{spn}}{\theta\rho^{sp}} (\xi\boldsymbol\om)^2 |A_n|.
\]
Collecting these estimates on the right-hand side of the energy estimate, we arrive at
\begin{align*}
\essup_{-\theta\tilde{\rho}^{sp}_n<t<0}&\int_{\widetilde{K}_n} w_-^2\,\dx +
\int_{-\theta\tilde{\rho}^{sp}_n}^{0}\int_{\widetilde{K}_n}\int_{\widetilde{K}_n}  \frac{|w_-(x,t) - w_-(y,t)|^p}{|x-y|^{N+sp}}\,\dx\dy\dt\\
&\le \boldsymbol\gm 2^{(N+2p)n}\frac{(\xi\boldsymbol\om)^p}{\dl\rho^{sp}}|A_n|.
\end{align*}

Now set $0\le\phi\le1$ to be a cutoff function in $\widetilde{Q}_n$, which vanishes outside $\overline{Q}_n$,
 equals the identity in $Q_{n+1}$ and satisfies $|D\phi|\le 2^n/\rho$.
 An application of the H\"older inequality  and the Sobolev imbedding
(cf.~Proposition~\ref{Sobolev-2} with $d=2^{-n-4}$)  gives that 
\begin{align*}
	&\frac{ \xi\boldsymbol\om}{2^{n+2}}
	|A_{n+1}|
	\le 
	\iint_{\widetilde{Q}_n}w_-\phi\,\dx\dt\\
	&\le
	\bigg[\iint_{\widetilde{Q}_n}\big(w_-\phi\big)^{p\kappa}
	\,\dx\dt\bigg]^{\frac{1}{p\kappa}}|A_n|^{1-\frac{1}{p\kappa}}\\
	&\le\boldsymbol\gm
	\bigg[\rho^{sp}\int_{-\theta\tilde{\rho}^p_n}^{0}\int_{\widetilde{K}_n}\int_{\widetilde{K}_n}  \frac{\big|w_-\phi(x,t) - w_-\phi(y,t)\big|^p}{|x-y|^{N+sp}}\,\dx\dy\dt+2^{(N+sp)n}\iint_{\widetilde{Q}_n}\big(w_-\phi\big)^p\,
	\dx\dt\bigg]^{\frac{1}{\kappa}}\\
	&\quad\ 
	\times\bigg[\essup_{-\theta\tilde{\varrho}_n^{sp}<t<0}
	\bint_{\widetilde{K}_n}\big(w_-\phi\big)^2\,\dx\bigg]^{\frac{(\kappa_*-1)}{\kappa_* \kappa p}}
	 |A_n|^{1-\frac{1}{\kappa p}}\\
	&\le 
	\boldsymbol\gm\boldsymbol b^n \dl^{-(\frac1{\kappa}+\frac{2\kappa_*-1}{\kappa_*\kappa p})} \rho^{-(N+sp)\frac{\kappa_* -1}{\kappa_*\kappa p}} (\xi\boldsymbol\om)^{\frac{2\kappa_* -1}{\kappa_* \kappa}}  
	|A_n|^{1+\frac{\kappa_* -1}{\kappa_*\kappa p}}
\end{align*}
for some $\boldsymbol b=\boldsymbol b(p,N)>1$.
To obtain the last line, we used the triangle inequality
\begin{align*}
\big|w_-\phi(x,t) &- w_-\phi(y,t)\big|^p \\
&\le c  \big|w_-(x,t) - w_-(y,t) \big|^p \phi^p(x,t) 
 + c  w^p_-(y,t)  \big|\phi(x,t) - \phi(y,t)\big|^p,
\end{align*}
for some $c=c(p)$, such that
\begin{align*}
\rho^{sp}&\int_{-\theta\tilde{\rho}^p_n}^{0}\int_{\widetilde{K}_n}\int_{\widetilde{K}_n}  \frac{\big|w_-\phi(x,t) - w_-\phi(y,t)\big|^p}{|x-y|^{N+sp}}\,\dx\dy\dt\\
&\le c \rho^{sp}\int_{-\theta\tilde{\rho}^p_n}^{0}\int_{\widetilde{K}_n}\int_{\widetilde{K}_n}  \frac{\big|w_-(x,t) - w_-(y,t)\big|^p}{|x-y|^{N+sp}}\,\dx\dy\dt\\
&\quad + c \rho^{sp}\int_{-\theta\tilde{\rho}^p_n}^{0}\int_{\widetilde{K}_n}\int_{\widetilde{K}_n}  \frac{w^p_-(y,t)\big|\phi(x,t) - \phi(y,t)\big|^p}{|x-y|^{N+sp}}\,\dx\dy\dt\\
&\le c \rho^{sp}\int_{-\theta\tilde{\rho}^p_n}^{0}\int_{\widetilde{K}_n}\int_{\widetilde{K}_n}  \frac{\big|w_-(x,t) - w_-(y,t)\big|^p}{|x-y|^{N+sp}}\,\dx\dy\dt\\
&\quad + \boldsymbol\gm 2^{pn} \iint_{\widetilde{Q}_n}  w^p_-(y,t) \,\dy\dt.
\end{align*}
Plugging this into the second-to-last line and employing the above energy estimate, the last line follows. 
Notice also from  Proposition~\ref{Sobolev-2} there holds
\[
\kappa:=1+\frac{2(\kappa_* -1)}{p\kappa_*}.
\]

In terms of $ \boldsymbol Y_n=|A_n|/|Q_n|$, this estimate leads to the recursive inequality
\begin{equation*}
\begin{aligned}
	 \boldsymbol Y_{n+1}
	\le
	&\boldsymbol\gm \dl^{-(\frac1{\kappa}+\frac{\kappa_*}{\kappa_*\kappa p})}(2\boldsymbol b)^n  \boldsymbol Y_n^{1+\frac{\kappa_* -1}{\kappa_*\kappa p}},
\end{aligned}
\end{equation*}
for some generic constant $\boldsymbol\gm$. 
Hence, by \cite[Chapter I, Lemma~4.1]{DB}, 
there exists
a positive constant $\nu$ depending only on the data,
 such that
$\boldsymbol Y_n\to0$ if we require that $\boldsymbol Y_o\le \nu$.
\end{proof}
The next lemma is a variant of the previous one, involving quantitative initial data.
\begin{lemma}\label{Lm:DG:initial:1}
Let $u$ be a locally bounded, local weak sub(super)-solution to \eqref{Eq:1:1} -- \eqref{Eq:K} in $E_T$. 
Let $\xi\in(0,1)$. 
There exists a positive constant $\nu_o$ depending only on the data, such that if
\[
\pm\big(\boldsymbol\mu^{\pm}-u(\cdot, t_o)\big)\ge \xi\boldsymbol\om \quad\text{ a.e. in } K_{\rho}(x_o),
\]
then either
\begin{equation*}
	 \Big( \frac{\rho}{R}\Big)^{\frac{sp}{p-1}} {\rm Tail}\big(\big(u - \boldsymbol \mu^{\pm}\big)_{\pm}; \mathcal{Q}\big) >\xi\boldsymbol\om,
\end{equation*}
or
\[
\pm\big(\boldsymbol\mu^{\pm}-u\big)\ge \tfrac12\xi\boldsymbol\om\quad\text{ a.e. in }K_{\frac12\rho}(x_o)\times\big(t_o, t_o+\nu_o(\xi\boldsymbol\om)^{2-p}\rho^{sp}\big],
\]
provided the cylinders are included in $\mathcal{Q}$.
\end{lemma}
\begin{proof}
Assume $(x_o,t_o)=(0,0)$. It suffices to show the case of super-solutions with $\boldsymbol\mu^-=0$. 
Let us first examine the energy estimate of Proposition~\ref{Prop:2:1}
in $Q(R,S)\equiv K_{\rho}\times(0,\theta\rho^{sp})$ for some $\theta$ to be determined.
Note that the time level $t_o-S$ in Proposition~\ref{Prop:2:1} corresponds to $t=0$ here.
Let $\z(x)$ be a time independent, piecewise smooth, cutoff function in $K_\rho$ that vanishes on $\pl K_\rho$.
If we take the level $k\le \xi\boldsymbol\om$, the spatial integral at $t=0$ 
(i.e. the term at the time level $t_o-S$ in Proposition~\ref{Prop:2:1}) vanishes due to
the assumption that $u(\cdot, 0)\ge   \xi\boldsymbol\om$ a.e. in $K_{\rho}$.
The term involving $\pl_t\z$ also vanishes since $\z$ is independent of $t$.
 As a result, the energy estimate reads
\begin{align*}
	\essup_{0<t<\theta\varrho^{sp}}&
	\int_{K_\varrho}\z^p w_-^2\,\dx
	+
	 \int_{0}^{\theta\varrho^{sp}}\int_{K_\rho}\int_{K_\rho}  \min\big\{\z^p(x),\z^p(y)\big\} \frac{|w_-(x,t) - w_-(y,t)|^p}{|x-y|^{N+sp}}\,\dx\dy\dt\\\
	&\quad\le
	 \boldsymbol\gm\int_{0}^{\theta\varrho^{sp}}\int_{K_\rho}\int_{K_\rho}\max\big\{w^p_{-}(x,t), w^p_{-}(y,t)\big\} \frac{|\z(x) - \z(y)|^p}{|x-y|^{N+sp}}\,\dx\dy\dt\\
	&\qquad+\boldsymbol\gm\int_{0}^{\theta\varrho^{sp}}\int_{K_\rho} \z^pw_{-}(x,t)\,\dx\dt \bigg(\essup_{\substack{x\in\supp\z \\t\in(0,\theta\varrho^{sp})}} \int_{\rn\setminus K_\rho}\frac{ w_{-}^{p-1}(y,t)}{|x-y|^{N+sp}}\,\dy\bigg).
\end{align*}

Introduce $k_n$,  $\rho_n$, $\tilde{\rho}_n$, $\hat{\rho}_n$, $\bar{\rho}_n$, $K_n$, $\widetilde{K}_n$, $\widehat{K}_n$ and $\overline{K}_n$ as in Lemma~\ref{Lm:DG:1}. 
The only difference is that the cylinders $Q_n$, $\widetilde{Q}_n$, $\widehat {Q}_n$ and $\overline{Q}_n$ are now of forward type, i.e. $Q_n=K_n\times(0,\theta\rho^{sp})$, $\widetilde{Q}_n=\widetilde{K}_n\times(0,\theta\rho^{sp})$, $\widehat{Q}_n=\widehat{K}_n\times(0,\theta\rho^{sp})$ and $\overline{Q}_n=\overline{K}_n\times(0,\theta\rho^{sp})$. Note that while shrinking
the base balls $K_n$, $\widehat{K}_n$, $\widetilde{K}_n$ and $\overline{K}_n$ along $\rho_n$,  the height of the cylinders is fixed.
For the piecewise smooth function $\z(x)$ in $K_n$,
we choose it to vanish outside $\widehat{K}_n$, be equal to $1$ in $\widetilde{K}_n$, and satisfy $|D\z|\le 2^n/\rho$.
By a similar treatment of the right-hand side as in Lemma~\ref{Lm:DG:1}, after enforcing
$$ \Big( \frac{\rho}{R}\Big)^{\frac{sp}{p-1}}   {\rm Tail}(u_{-}; \mathcal{Q})  \le \xi\boldsymbol\om,$$
 we may obtain that
\begin{align*}
\essup_{0<t<\theta\rho^{sp} }&\int_{\widetilde{K}_n} w_-^2\,\dx +
\int^{\theta\rho^{sp}}_{0}\int_{\widetilde{K}_n}\int_{\widetilde{K}_n}  \frac{|w_-(x,t) - w_-(y,t)|^p}{|x-y|^{N+sp}}\,\dx\dy\dt\\
&\le \boldsymbol\gm 2^{(N+2p)n}\frac{(\xi\boldsymbol\om)^p}{\rho^{sp}}|A_n|,
\end{align*}
where  $A_n:=\{u<k_n\}\cap Q_n$.

Then we may proceed as in Lemma~\ref{Lm:DG:1} to obtain the recursive inequality
\begin{equation*}
\begin{aligned}
	 \boldsymbol Y_{n+1}
	\le
	&\boldsymbol\gm\boldsymbol b^n  \bigg[\frac{\theta}{(\xi\boldsymbol\om)^{2-p}}\bigg]^{\frac{\kappa_* -1}{\kappa_*\kappa p}}\boldsymbol Y_n^{1+\frac{\kappa_* -1}{\kappa_*\kappa p}},
\end{aligned}
\end{equation*}
where $ \boldsymbol Y_n=|A_n|/|Q_n|$ and the constants $\boldsymbol\gm,\,\boldsymbol b$ depend only on the data.
Hence, by \cite[Chapter I, Lemma~4.1]{DB}, 
there exists
a positive constant $\nu_o$ depending only on the data, such that if
$$\boldsymbol Y_o\le   \frac{\nu_o(\xi\boldsymbol\om)^{2-p}}{\theta},$$ 
then
$\boldsymbol Y_n\to0$. 
To finish the proof, we choose $\theta=\nu_o(\xi\boldsymbol\om)^{2-p}$.
\end{proof}

The following lemma propagates measure theoretical information forward in time.

\begin{lemma}\label{Lm:3:1}
 Let $u$ be a locally bounded, local weak sub(super)-solution to \eqref{Eq:1:1} -- \eqref{Eq:K} in $E_T$.
Introduce parameters $\xi$ and $\al$ in $(0,1)$. There exist $\dl$ and $\varep$ in $(0,1)$, depending only on the data and $\al$, such that if
	\begin{equation*}
	\Big|\Big\{
		\pm\big(\boldsymbol \mu^{\pm}-u(\cdot, t_o)\big)\ge \xi\boldsymbol \om
		\Big\}\cap K_{\varrho}(x_o)\Big|
		\ge\al \big|K_{\varrho}\big|,
	\end{equation*}
	then either $$ \Big( \frac{\rho}{R}\Big)^{\frac{sp}{p-1}} {\rm Tail}\big(\big(u - \boldsymbol \mu^{\pm}\big)_{\pm}; \mathcal{Q}\big) >\xi\boldsymbol \om,$$ or
	\begin{equation*}
	\Big|\Big\{
	\pm\big(\boldsymbol \mu^{\pm}-u(\cdot, t)\big)\ge \varep \xi\boldsymbol \om\Big\} \cap K_{\varrho}(x_o)\Big|
	\ge\frac{\al}2 |K_\varrho|
	\quad\mbox{ for all $t\in\big(t_o,t_o+\dl(\xi\boldsymbol \om)^{2-p}\varrho^{sp}\big]$,}
\end{equation*}
provided this cylinder is included in $\mathcal{Q}$. Moreover, we may trace the dependences by $\varep\approx\al$ and $\dl\approx\al^{p+1}$.
\end{lemma}
\begin{proof}
 We only show the case of super-solutions with $\boldsymbol\mu^-=0$.
Assume $(x_o,t_o)=(0,0)$. 
Use the energy estimate in Proposition~\ref{Prop:2:1}
in the cylinder $Q=K_{\varrho}\times(0,\dl(\xi\boldsymbol \om)^{2-p}\varrho^p]$, with
$k=\xi\boldsymbol\om$ and choose a standard non-negative cutoff function 
$\z(x,t)\equiv \z(x)$ independent of time that equals $1$ on $K_{(1-\sig)\varrho}$ with $\sigma\in(0,1)$ to be chosen later 
and vanishes on $\pl K_{\varrho}$ satisfying
$|D\z|\le(\sig\varrho)^{-1}$;
in such a case, estimating the right-hand side as in Lemma~\ref{Lm:DG:1} and enforcing
$$ \Big( \frac{\rho}{R}\Big)^{\frac{sp}{p-1}}   {\rm Tail}(u_{-}; \mathcal{Q})  \le \xi\boldsymbol\om,$$
we have for all $0<t<\dl(\xi\boldsymbol \om)^{2-p}\varrho^p$, that
\begin{align*}
	\int_{K_{\sig\varrho}\times\{t\}} (u-\xi\boldsymbol\om)^2_- \,\dx\nonumber
	\le
	\int_{K_\varrho\times\{0\}} (u-\xi\boldsymbol\om)^2_-\, \dx
	+
	\boldsymbol \gm\frac{  \dl(\xi\boldsymbol\om)^{2}}{\sig^p}|K_\varrho|. 
\end{align*}
The rest of the proof is standard and we refer to \cite[Chapter~4, Lemma~1.1]{DBGV-mono}.
\end{proof}

The following measure shrinking lemma is usually a delicate part in the  theory of the  parabolic $p$-Laplacian.
However, the  term that involves mixed positive/negative truncations  in the energy estimate greatly simplifies the argument. 
For ease of notation, we omit the vertex $(x_o,t_o)$ from $Q_\rho(\theta)$. 
\begin{lemma}\label{Lm:3:2}
 Let $u$ be a locally bounded, local weak sub(super)-solution to \eqref{Eq:1:1} -- \eqref{Eq:K} in $E_T$.
Suppose that for  some $\dl$, $\sig$ and $\xi$ in $(0,\tfrac12)$, there holds
	\begin{equation*}
	\Big|\Big\{
		\pm\big(\boldsymbol \mu^{\pm}-u(\cdot, t)\big)\ge \xi\boldsymbol \om
		\Big\}\cap K_{\varrho}(x_o)\Big|
		\ge\al \big|K_{\varrho}\big|\quad\mbox{ for all $t\in\big(t_o-\dl(\sig\xi\boldsymbol \om)^{2-p}\varrho^{sp}, t_o\big]$,}
	\end{equation*}
Let $\theta=\dl(\sig\xi\boldsymbol \om)^{2-p}$.
There exists
 $\boldsymbol \gm>0$ depending only on the data,  such that 
  either $$ \Big( \frac{\rho}{R}\Big)^{\frac{sp}{p-1}} {\rm Tail}\big(\big(u - \boldsymbol \mu^{\pm}\big)_{\pm}; \mathcal{Q}\big) >\sig\xi\boldsymbol\om,$$ or
\begin{equation*}
	\Big|\Big\{
	\pm\big(\boldsymbol \mu^{\pm}-u\big)\le \sig \xi\boldsymbol \om \Big\}\cap Q_{\rho}(\theta)\Big|
	\le \boldsymbol\gm \frac{\sig^{p-1}}{\dl\al} |Q_{\rho}(\theta)|,
\end{equation*}
provided $Q_{2\rho}(\theta)$ is included in $\mathcal{Q}$.
\end{lemma}
\begin{proof}
It suffices to show the case of super-solutions with $\boldsymbol \mu^-=0$.
For simplicity, we assume $(x_o,t_o)=(0,0)$.
 We  employ the energy estimate of Proposition~\ref{Prop:2:1} in $K_{2\rho}\times (-\theta\rho^p, 0]$ 
with the truncation
\begin{equation*}
	w_- =(u-k)_- \quad \text{and}\quad k= \sig \xi\boldsymbol \om,
\end{equation*}
 and introduce
 a  cutoff function $\z$ in $K_{ 2\varrho}$ (independent of $t$) that is equal to $1$ in $K_{\varrho}$ and vanishes
outside $K_{\frac32\varrho}$, such that $|D\z|\le\varrho^{-1}$.
Then, we obtain from Proposition~\ref{Prop:2:1} that
\begin{align*}
	\iint_{Q_\rho(\theta)}& w_{-}(y,t) \,\dy\dt \bigg(\int_{ K_{2\rho}}  \frac{w^{p-1}_+(x,t)}{|x-y|^{N+sp}}\,\dx\bigg)\\
	& \le
	\boldsymbol\gm\int_{-\theta\rho^p}^{0}\int_{K_{2\rho}}\int_{K_{2\rho}}\max\big\{w_{-}(x,t), w_{-}(y,t)\big\}^p \frac{|\z(x) - \z(y)|^p}{|x-y|^{N+sp}}\,\dx\dy\dt\\
	&\quad+\boldsymbol\gm\iint_{Q_{2\rho}(\theta)} \z^p w_{-}(x,t)\,\dx\dt \bigg(\essup_{\substack{x\in K_{\frac32\rho}\\t\in(-\theta\rho^p,0)}} \int_{\rn\setminus K_{2\rho}}\frac{ w_{-}^{p-1}(y,t)}{|x-y|^{N+sp}}\,\dy\bigg)\\
	&\quad +
	\int_{K_{2\varrho} } w_-^2(x,-\theta\rho^p) \,\dx.
\end{align*}

The various terms on the right-hand side are estimated as in Lemma~\ref{Lm:DG:1}. Indeed, after enforcing
$$ \Big( \frac{\rho}{R}\Big)^{\frac{sp}{p-1}}  {\rm Tail}(u_{-}; \mathcal{Q})  \le \sig\xi\boldsymbol\om,$$
they are bounded by
\[
\boldsymbol\gm \frac{(\sig\xi\boldsymbol\om)^p}{\dl\rho^{sp}}|Q_{\rho}(\theta)|.
\]

The left-hand side is estimated by extending the integrals over smaller sets and by using the given measure theoretical information:
\begin{align*}
\iint_{Q_\rho(\theta) }&w_{-}(y,t) \chi_{\{u(y,t)\le\frac12\sig\xi\boldsymbol\om\}} \,\dy\dt \bigg(\int_{ K_{2\rho}}  \frac{w^{p-1}_+(x,t)\chi_{\{u(x,t)\ge\xi\boldsymbol\om\} }}{|x-y|^{N+sp}}\,\dx\bigg)\\
&\ge \tfrac12\sig\xi\boldsymbol\om\Big|\Big\{u\le\tfrac12\sig\xi\boldsymbol\om\Big\}\cap Q_\rho(\theta)\Big| \bigg( \frac{  (\tfrac12\xi\boldsymbol\om)^{p-1}\al |K_\rho|}{(4\rho)^{N+sp}} \bigg)\\
&=\boldsymbol\gm\frac{(\xi\boldsymbol\om)^p\al\sig}{\rho^{sp}} \Big|\Big\{u\le\tfrac12\sig\xi\boldsymbol\om\Big\}\cap Q_\rho(\theta)\Big|.
\end{align*}

Combining these estimates and properly adjusting relevant constants, we conclude the proof.
\end{proof}


\section{Proof of Theorem~\ref{Thm:1:1}: $1<p\le 2$}

\subsection{Expansion of Positivity}
Suppose the cylinder $\mathcal{Q}$ and the numbers $\boldsymbol \mu^{\pm}$ and $\boldsymbol \om$ are defined as in Section~\ref{S:prelim}. The key ingredient of the reduction of oscillation lies in the following  expansion of positivity valid for $1<p\le 2$.
\begin{proposition}\label{Prop:expansion}
Let $u$ be a locally bounded, local, weak sub(super)-solution to \eqref{Eq:1:1} -- \eqref{Eq:K} in $E_T$.
Suppose for some constants  $\al,\xi \in(0,1)$, there holds
	\begin{equation*}
		\Big|\Big\{\pm\big(\boldsymbol \mu^{\pm}-u(\cdot, t_o)\big)\ge \xi\boldsymbol\om \Big\}\cap K_{\varrho}(x_o) \Big|
		\ge
		\al \big|K_\varrho \big|.
	\end{equation*}
Then there exist constants $\dl,\,\eta\in(0,1)$ depending only on the data and $\al$, such that 
either
\begin{equation*}
	\Big( \frac{\rho}{R}\Big)^{\frac{sp}{p-1}} {\rm Tail}\big(\big(u - \boldsymbol \mu^{\pm}\big)_{\pm}; \mathcal{Q}\big) >\eta \xi \boldsymbol\om,
\end{equation*}
 or
\begin{equation*}
	\pm\big(\boldsymbol \mu^{\pm}-u\big)\ge\eta\xi\boldsymbol\om
	\quad
	\mbox{a.e.~in $ K_{2\varrho}(x_o) \times\big( t_o+\tfrac12 \dl (\xi\boldsymbol\om)^{2-p}\varrho^{sp},
	t_o+\dl (\xi\boldsymbol\om)^{2-p}\varrho^{sp}\big],$}
\end{equation*}
provided 
\[
K_{4\rho}(x_o)\times\big(t_o, t_o+\dl (\xi\boldsymbol\om)^{2-p}\varrho^{sp}\big]\subset \mathcal{Q}.
\]
Moreover, we have  $\dl\approx\al^{p+1}$ and $\eta\approx\al^q$ for some $q>1$ depending on the data.
\end{proposition}
\begin{proof}
Assuming $(x_o,t_o)=(0,0)$ and $\boldsymbol\mu^{-}=0$ for simplicity, it suffices to deal with super-solutions.
 Rewriting the measure theoretical information at the initial time $t_o=0$ in the larger ball $K_{4\rho}$ and replacing $\al$ by $4^{-N}\al$, we can enforce
 $$ \Big( \frac{\rho}{R}\Big)^{\frac{sp}{p-1}}   {\rm Tail}(u_{-}; \mathcal{Q}) \le\xi\boldsymbol \om,$$
 and apply Lemma~\ref{Lm:3:1} to obtain $\dl, \varep\in(0,1)$, such that 
	\begin{equation*}
	\Big|\Big\{
	 u(\cdot, t) \ge \varep \xi\boldsymbol \om\Big\} \cap K_{4\varrho} \Big|
	\ge\frac{\al}2 4^{-N} |K_{4\varrho}|
	\quad\mbox{ for all $t\in\big(0, \dl(\xi\boldsymbol \om)^{2-p}(4\varrho)^{sp}\big]$.}
\end{equation*}

This measure theoretical information for each slice of the time interval in turn allows us to apply Lemma~\ref{Lm:3:2} with $\xi$ and $\al$ there replaced by $\varep\xi$ and $\tfrac12 4^{-N}\al$, since $\sig\in(0,1)$ and 
\[
\dl(\sig\varep\xi\boldsymbol \om)^{2-p}\le \dl(\xi\boldsymbol \om)^{2-p}. 
\]
Note that the above inequality used the fact that $p\le2$.
Letting $\nu$ be determined in Lemma~\ref{Lm:DG:1}  in terms of the data and $\dl$, we further choose $\sig$ according to Lemma~\ref{Lm:3:2} to satisfy
\[
\boldsymbol\gm \frac{\sig^{p-1}}{\dl\al} <\nu, \quad\text{i.e.}\quad \sig\le\Big(\frac{\nu\dl\al}{\boldsymbol\gm}\Big)^{\frac1{p-1}}.
\]
Enforcing
\begin{equation*}
	\Big( \frac{\rho}{R}\Big)^{\frac{sp}{p-1}} {\rm Tail}\big(  u_{-}; \mathcal{Q}\big) \le \sig\varep \xi \boldsymbol\om,
\end{equation*}
such a choice of $\sig$ permits us to apply Lemma~\ref{Lm:DG:1} and conclude that
\[
u\ge\tfrac12\sig\varep\xi\boldsymbol\om\quad\text{a.e. in}\>K_{2\varrho}  \times\big(  \dl (\xi\boldsymbol\om)^{2-p}(4\varrho)^{sp}-\dl(\sig\varep\xi\boldsymbol\om)^{2-p}(2\rho)^{sp}, \dl (\xi\boldsymbol\om)^{2-p}(4\varrho)^{sp}\big].
\]
The proof is completed by defining $\eta=\sig\varep$ and properly adjusting relevant constants in dependence of the data.
\end{proof}

Based on Proposition~\ref{Prop:expansion}, the remaining part is devoted to the proof of Theorem~\ref{Thm:1:1} for $1<p\le 2$. All constants determined in the course of the proof are stable as $p\to2$.
\subsection{The First Step} 
For some $Q_{\widetilde{R}}\subset E_T$ we introduce
\[
\boldsymbol\om=2\essup_{Q_{\widetilde{R}}} |u| +{\rm Tail}(u; Q_{\widetilde{R}})
\]
and $Q_o=Q_R(\boldsymbol\om^{2-p})$. 
By properly shrinking $R$, we may assume that $Q_o\subset Q_{\widetilde{R}}$ and
set
\begin{equation*}
	\boldsymbol \mu^+=\essup_{Q_o}u,
	\qquad
	\boldsymbol\mu^-=\essinf_{Q_o}u.
\end{equation*}
Without loss of generality, we take $(x_o,t_o)=(0,0)$.

Then the following intrinsic relation holds true:
\begin{equation}\label{Eq:start-cylinder}
 \essosc_{Q_R(\boldsymbol\om^{2-p})}u\le\boldsymbol\om.
\end{equation}
The choice of the reference cylinder $Q_{\widetilde{R}}$ is made to verify 
\eqref{Eq:start-cylinder}, on which the subsequent arguments are based. 

Let $\dl \in(0,1)$ be determined in Proposition~\ref{Prop:expansion} with $\al=\tfrac12$. 
For some $c\in(0,\tfrac14)$ to be chosen, define $$\tau:=\dl(\tfrac14 \boldsymbol\om)^{2-p}( c R)^{sp}$$ and consider two alternatives
\begin{equation*}
\left\{
\begin{array}{ll}
	\Big|\Big\{u\big(\cdot,-\tau\big)-\boldsymbol\mu^->\tfrac14 \boldsymbol\om\Big\}
	\cap 
	K_{c R}\Big|
	\ge
	\tfrac12 |K_{c R}|,\\[5pt]

	\Big|\Big\{\boldsymbol\mu^+ - u\big(\cdot,-\tau\big)>\tfrac14 \boldsymbol\om\Big\}
	\cap 
	K_{c R}\Big|
	\ge
	\tfrac12 |K_{c R}|.
\end{array}\right.
\end{equation*}
Assuming $\boldsymbol\mu^+ - \boldsymbol\mu^-\ge\tfrac12\boldsymbol\om$, one of the two alternatives must hold.
Whereas the case $\boldsymbol\mu^+ - \boldsymbol\mu^-<\tfrac12\boldsymbol\om$ will be trivially incorporated into the forthcoming oscillation estimate \eqref{Eq:osc:0}.

Let us suppose the first alternative holds for instance. An appeal to
Proposition~\ref{Prop:expansion} with $\al=\tfrac12$, $\xi=\tfrac14$ and $\rho=cR$ determines $\eta\in(0,\tfrac12)$ and yields that
either
\begin{equation}\label{Eq:tail:0}
	c^{\frac{sp}{p-1}} {\rm Tail}\big( \big(u-\boldsymbol \mu^{-}\big)_{-}; Q_o\big) > \eta  \boldsymbol\om,
\end{equation}
 or
\begin{equation*}
	u-\boldsymbol \mu^{-} \ge\eta\boldsymbol\om
	\quad\text{a.e. in}\> Q_{cR}(\dl(\tfrac14 \boldsymbol\om)^{2-p}),
\end{equation*}
which, thanks to \eqref{Eq:start-cylinder}, gives the reduction of oscillation
\begin{equation}\label{Eq:osc:0}
\essosc_{Q_{cR}(\dl(\frac14 \boldsymbol\om)^{2-p})} u\le \big(1-\eta)\boldsymbol\om=:\boldsymbol\om_1.
\end{equation}
The number $c$ is chosen to ensure that \eqref{Eq:tail:0} does not happen. Indeed, we may first estimate 
\begin{equation}\label{Eq:tail:1}
{\rm Tail}\big( \big(u-\boldsymbol \mu^{-}\big)_{-}; Q_o\big) \le \boldsymbol\gm\boldsymbol\om.
\end{equation}
This can be seen by the definitions of $\boldsymbol\om$ and the tail,
\begin{align*}
\big[{\rm Tail}&\big( \big(u-\boldsymbol \mu^{-}\big)_{-}; Q_o\big)\big]^{p-1} =R^{sp} \essup_{-\boldsymbol\om^{2-p}R^{sp}<t<0}\int_{\rn\setminus K_{R}}\frac{\big(u-\boldsymbol \mu^{-}\big)_{-}^{p-1}}{|x|^{N+sp}}\,\dx\\
&\le \boldsymbol\gm \boldsymbol\om^{p-1}+\boldsymbol\gm R^{sp}
\essup_{-\boldsymbol\om^{2-p}R^{sp}<t<0}
  \int_{\rn\setminus K_R} \frac{u_{-}^{p-1}}{|x|^{N+sp}}\,\dx\\
&=\boldsymbol\gm \boldsymbol\om^{p-1}+\boldsymbol\gm R^{sp} \essup_{-\boldsymbol\om^{2-p}R^{sp}<t<0}
\bigg[ \int_{\rn\setminus K_{\widetilde{R}}} \frac{u_{-}^{p-1}}{|x|^{N+sp}}\,\dx+
 \int_{K_{\widetilde{R}}\setminus K_{R}} \frac{u_{-}^{p-1}}{|x|^{N+sp}}\,\dx\bigg]\\
&\le\boldsymbol\gm\boldsymbol\om^{p-1}.
\end{align*}
Then, using \eqref{Eq:tail:1} we choose
\begin{equation}\label{Eq:c:0}
c^{\frac{sp}{p-1}}\boldsymbol\gm\boldsymbol\om\le \eta  \boldsymbol\om,\quad\text{i.e.}\quad c\le \Big(\frac{\eta}{\boldsymbol\gm}\Big)^{\frac{p-1}{sp}},
\end{equation}
such that \eqref{Eq:tail:0} does not occur.
Note that \eqref{Eq:c:0} is not the final choice of $c$ yet and it is subject to a further smallness requirement. 

Next we set $R_1=\lm R$ for some $\lm\le c$ to verify the set inclusion
\begin{equation}\label{Eq:lm:0}
Q_{R_1}(\boldsymbol\om_1^{2-p}) \subset Q_{cR}(\dl(\tfrac14 \boldsymbol\om)^{2-p}),\quad\text{i.e.}\quad \lm\le 4^{\frac{p-2}{p}}\dl^{\frac1p}c.
\end{equation}
As a result of this inclusion and \eqref{Eq:osc:0} we obtain
\begin{equation*}
 \essosc_{Q_{R_1}(\boldsymbol\om_1^{2-p})}u\le\boldsymbol\om_1,
\end{equation*}
which plays the role of \eqref{Eq:start-cylinder} in the next stage.

\subsection{The Induction}
Now we may proceed by induction. 
 Suppose up to $i=0,1,\cdots, j$, we have built
\begin{equation*}
\left\{
	\begin{array}{c}
	\dsty R_o=R,\quad R_i=\lm R_{i-1},
	\quad 
	\boldsymbol\om_i=(1-\eta)\boldsymbol\om_{i-1},
	\quad 
	Q_i=Q_{R_i}(\boldsymbol\om_i^{2-p}), 
	\\[5pt]
	\dsty\boldsymbol\mu_i^+=\essup_{Q_i}u,
	\quad
	\boldsymbol\mu_i^-=\essinf_{Q_i}u,
	\quad
	\essosc_{Q_i}u\le\boldsymbol\om_i.
	\end{array}
\right.
\end{equation*}
The induction argument will show that the above oscillation estimate continues to hold for the $(j+1)$-th step.

Let  $\dl$ be fixed as before, whereas $c\in(0,1)$ is subject to a further choice. To reduce the oscillation in the next stage, we basically repeat what has been done in the first step, now with $\boldsymbol\mu^{\pm}_j$, $\boldsymbol\om_j$, $R_j$, $Q_j$, etc. 
In fact, we define $$\tau:=\dl(\tfrac14 \boldsymbol\om_j)^{2-p}( c R_j)^{sp}$$ and consider two alternatives
\begin{equation*}
\left\{
\begin{array}{ll}
	\Big|\Big\{u\big(\cdot,-\tau\big)-\boldsymbol\mu^-_j>\tfrac14 \boldsymbol\om_j\Big\}
	\cap 
	K_{c R_j}\Big|
	\ge
	\tfrac12 |K_{c R_j}|,\\[5pt]

	\Big|\Big\{\boldsymbol\mu^+_j - u\big(\cdot,-\tau\big)>\tfrac14 \boldsymbol\om_j\Big\}
	\cap 
	K_{c R_j}\Big|
	\ge
	\tfrac12 |K_{c R_j}|.
\end{array}\right.
\end{equation*}
Like in the first step, we may assume $\boldsymbol\mu^+_j - \boldsymbol\mu^-_j\ge\tfrac12\boldsymbol\om_j$, so that one of the two alternatives must hold.
Otherwise the case $\boldsymbol\mu^+_j - \boldsymbol\mu^-_j<\tfrac12\boldsymbol\om_j$ can be trivially incorporated into the forthcoming oscillation estimate \eqref{Eq:osc:j}.

Let us suppose the first case holds for instance. An application of
Proposition~\ref{Prop:expansion} in $Q_j$, with $\al=\tfrac12$, $\xi=\tfrac14$ and $\rho=cR_j$ yields (for the same $\eta$ as before) that
either
\begin{equation}\label{Eq:tail:j}
	c^{\frac{sp}{p-1}} {\rm Tail}\big( \big(u-\boldsymbol \mu_j^{-}\big)_{-}; Q_j\big) > \eta  \boldsymbol\om_j,
\end{equation}
 or
\begin{equation*}
	u-\boldsymbol \mu_j^{-} \ge\eta\boldsymbol\om_j
	\quad\text{a.e. in}\> Q_{cR_j}(\dl(\tfrac14 \boldsymbol\om_j)^{2-p}),
\end{equation*}
which, thanks to the $j$-th induction assumption, gives the reduction of oscillation
\begin{equation}\label{Eq:osc:j}
\essosc_{Q_{cR_j}(\dl(\frac14 \boldsymbol\om_j)^{2-p})} u\le \big(1-\eta\big)\boldsymbol\om_j=:\boldsymbol\om_{j+1}.
\end{equation}

The final choice of $c$ is made to ensure that \eqref{Eq:tail:j} does not happen, independent of $j$. This hinges upon the following tail estimate
\begin{equation}\label{Eq:tail-est}
 {\rm Tail}\big( \big(u-\boldsymbol \mu_j^{-}\big)_{-}; Q_j\big)  \le \boldsymbol\gm\boldsymbol\om_j.
\end{equation}
 To prove this, we first rewrite the tail as follows:
\begin{align*}
\big[{\rm Tail}&\big( \big(u-\boldsymbol \mu_j^{-}\big)_{-}; Q_j\big)\big]^{p-1}=R_j^{sp}\essup_{-\boldsymbol\om_j^{2-p}R_j^{sp}<t<0}\int_{\rn\setminus K_j} \frac{\big(u-\boldsymbol \mu_j^{-}\big)_{-}^{p-1}}{|x|^{N+sp}}\,\dx\\
&=R_j^{sp}\essup_{-\boldsymbol\om_j^{2-p}R_j^{sp}<t<0} \bigg[ \int_{\rn\setminus K_R} \frac{\big(u-\boldsymbol \mu_j^{-}\big)_{-}^{p-1}}{|x|^{N+sp}}\,\dx+
\sum_{i=1}^{j}\int_{K_{i-1}\setminus K_{i}} \frac{\big(u-\boldsymbol \mu_j^{-}\big)_{-}^{p-1}}{|x|^{N+sp}}\,\dx\bigg].
\end{align*}
The first integral is estimated by using the definition of $\boldsymbol \om$. Namely, for any $t\in (-\boldsymbol\om_j^{2-p}R_j^{sp},0)$,
\begin{align*}
\int_{\rn\setminus K_R} \frac{\big(u-\boldsymbol \mu_j^{-}\big)_{-}^{p-1}}{|x|^{N+sp}}\,\dx
&\le 
\boldsymbol \gm\int_{\rn\setminus K_R} \frac{|\boldsymbol \mu_j^{-}|^{p-1}+u_{-}^{p-1}}{|x|^{N+sp}}\,\dx\\
&\le\boldsymbol \gm\frac{\boldsymbol \om^{p-1}}{R^{sp}}+\boldsymbol \gm\int_{K_{\widetilde{R}}\setminus K_R} \frac{u_{-}^{p-1}}{|x|^{N+sp}}\,\dx+\boldsymbol \gm\int_{\rn\setminus K_{\widetilde{R}}} \frac{u_{-}^{p-1}}{|x|^{N+sp}}\,\dx\\
&\le\boldsymbol \gm\frac{\boldsymbol \om^{p-1}}{R^{sp}}.
\end{align*}
Whereas the second integral is estimated by using the simple fact that, for $i=1,2,\cdots,j$,
$$\big(u-\boldsymbol \mu_j^{-}\big)_{-}\le \boldsymbol \mu_j^{-} - \boldsymbol \mu_{i-1}^{-}\le\boldsymbol \mu_j^{+} - \boldsymbol \mu_{i-1}^{-}\le\boldsymbol \mu_{i-1}^{+} - \boldsymbol \mu_{i-1}^{-} \le \boldsymbol \om_{i-1}\quad\text{a.e. in}\>Q_{i-1}.$$
 Namely, for any $t\in (-\boldsymbol\om_j^{2-p}R_j^{sp},0)$,
\begin{align*}
\int_{K_{i-1}\setminus K_{i}} \frac{\big(u-\boldsymbol \mu_j^{-}\big)_{-}^{p-1}}{|x|^{N+sp}}\,\dx
\le \boldsymbol \gm \frac{\boldsymbol \om_{i-1}^{p-1}}{R_i^{sp}}.
\end{align*} 
Combine them and further estimate the tail by 
\begin{equation*}
\begin{aligned}
\big[{\rm Tail}\big( \big(u-\boldsymbol \mu_j^{-}\big)_{-}; Q_j\big)\big]^{p-1}&\le\boldsymbol \gm R_j^{sp} \sum_{i=1}^{j}  \frac{\boldsymbol \om_{i-1}^{p-1}}{R_i^{sp}}\\
&=\boldsymbol \gm \boldsymbol \om_{j}^{p-1} \sum_{i=1}^{j} (1 - \eta)^{(j-i+1)(1-p)}\lm^{(j-i)sp}.
\end{aligned}
\end{equation*}
The summation in the last display is bounded by $(1-\eta)^{1-p}$ if we restrict the choice of $\lm$ by $$(1-\eta)^{1-p}\lm^{sp}\le\tfrac12,\quad\text{i.e.}\quad\lm\le 2^{-\frac1{sp}}(1-\eta)^{\frac{p-1}{sp}}.$$
Consequently, the tail estimate \eqref{Eq:tail-est} is proven and \eqref{Eq:tail:j} does not happen, if we choose $c$ to verify
\begin{equation}\label{Eq:c:j}
c^{sp}\boldsymbol \gm \le \eta^{p-1},\quad\text{i.e.}\quad c\le \frac{1}{\boldsymbol\gm}\eta^{\frac{p-1}{sp}}.
\end{equation}
The final choice of $c$ is made out of the smaller one of \eqref{Eq:c:0} and \eqref{Eq:c:j}.

Let $R_{j+1}=\lm R_j$ for some $\lm\in(0,1)$ to verify the set inclusion
\begin{equation}\label{Eq:lm:j}
Q_{R_{j+1}}(\boldsymbol\om_{j+1}^{2-p}) \subset Q_{cR_j}(\dl(\tfrac12 \boldsymbol\om_j)^{2-p}),\quad\text{i.e.}\quad \lm\le 2^{\frac{p-2}{p}}\dl^{\frac1p}c.
\end{equation}
Note that the choice of $\lm$ in the last display may have been adjusted from the precious one in \eqref{Eq:lm:0} due to the possible change of $c$ made in \eqref{Eq:c:j}.
The final choice of $\lm$ is
\[
\lm=\min\Big\{ 2^{-\frac1{sp}}(1-\eta)^{\frac{p-1}{sp}}, 2^{\frac{p-2}{p}}\dl^{\frac1p}c\Big\}.
\]
As a result of the inclusion \eqref{Eq:lm:j} and \eqref{Eq:osc:j} we obtain
\begin{equation*}
 \essosc_{Q_{R_{j+1}}(\boldsymbol\om_{j+1}^{2-p})}u\le\boldsymbol\om_{j+1},
\end{equation*}
which completes the induction argument. From now on, the deduction of a H\"older modulus of continuity becomes quite standard; cf.~\cite[Chapter~III, Proposition~3.1]{DB}.

\section{Proof of Theorem~\ref{Thm:1:1}: $p> 2$}
As in the previous section, we first introduce  $Q_{\widetilde{R}}\subset E_T$,
\[
\boldsymbol\om=2\essup_{Q_{\widetilde{R}}} |u| +{\rm Tail}(u; Q_{\widetilde{R}})
\]
and $Q_o=Q_R(a\theta)$ for $\theta=(\frac14\boldsymbol\om)^{2-p}$ and some $a\in(0,1)$ to be determined. 
By properly shrinking $R$, we may assume that $Q_o\subset Q_{\widetilde{R}}$ and
set
\begin{equation*}
	\boldsymbol \mu^+=\essup_{Q_o}u,
	\qquad
	\boldsymbol\mu^-=\essinf_{Q_o}u.
\end{equation*}
Without loss of generality, we take $(x_o,t_o)=(0,0)$.

Then the following intrinsic relation holds:
\begin{equation}\label{Eq:start-cylinder:1}
\essosc_{Q_{R}(a(\frac14\boldsymbol\om)^{2-p})}u\le\boldsymbol\om.
\end{equation}
As before, the choice of $Q_{\widetilde{R}}$ is made to ensure \eqref{Eq:start-cylinder:1}, on which the subsequent arguments are based. 

Unlike the case $1<p\le2$, an expansion of positivity for the case $p>2$ requires addition technical complication. To deal with this case, we instead refine DiBenedetto's argument in \cite{DB86}: On one hand, the tail needs a great care in this intrinsic scaling scenario; on the other hand, we dispense with any kind of logarithmic estimate and just rely on the energy estimate.

\subsection{The First Alternative}\label{S:case-1}
In this section, 
we work with $u$ as a super-solution near its
infimum.
Furthermore, we assume 
\begin{equation}\label{Eq:mu-pm-}
	\boldsymbol\mu^+ -\boldsymbol\mu^- >\tfrac12\boldsymbol\om.
\end{equation}
The other case $\boldsymbol\mu^+ -\boldsymbol\mu^- \le\frac12\boldsymbol\om$,
will be considered later.

Suppose  $a,\, c\in(0,1)$ verify that $a>2c^{sp}$ for the moment (which will be confirmed in \eqref{Eq:a} and \eqref{Eq:c2:1.5}), 
and for some $\bar{t}\in\big(-a\theta R^{sp}+\theta(cR)^{sp},0\big]$, there holds
\begin{equation}\label{Eq:1st-alt-meas}
	\Big|\Big\{u\le\boldsymbol\mu^-+\tfrac14 \boldsymbol\om\Big\}
	\cap 
	(0,\bar{t})+Q_{cR}(\theta)\Big|\le \nu|Q_{cR}(\theta)|,
\end{equation}
where $\nu$ is the constant determined in Lemma~\ref{Lm:DG:1} (with $\dl=1$) in terms of the data.
According to Lemma~\ref{Lm:DG:1}  with   $\dl=1$, $\xi=\frac14$ and $\rho=cR$, we have either
\begin{equation}\label{Eq:tail2:0}
	c^{\frac{sp}{p-1}} {\rm Tail}\big( \big(u-\boldsymbol \mu^{-}\big)_{-}; Q_o\big) > \tfrac14   \boldsymbol\om,
\end{equation}
or
\begin{equation}\label{Eq:lower-bd}
	u\ge\boldsymbol \mu^-+\tfrac{1}8\boldsymbol\om
	\quad
	\mbox{a.e.~in $(0,\bar{t})+Q_{\frac12 cR}(\theta)$.}
\end{equation}

To proceed, we restrict $c$ so that \eqref{Eq:tail2:0} does not happen. Indeed, since, according to the definition of $\boldsymbol\om$, the tail can be easily estimated by (cf.~\eqref{Eq:tail:1})
\[
{\rm Tail}\big( \big(u-\boldsymbol \mu^{-}\big)_{-}; Q_o\big) \le \boldsymbol\gm\boldsymbol\om,
\]
we impose
\begin{equation}\label{Eq:c2:0}
c^{\frac{sp}{p-1}} \boldsymbol\gm\boldsymbol\om\le\tfrac14\boldsymbol\om,\quad\text{i.e.}\quad c\le\Big(\frac1{4\boldsymbol\gm}\Big)^{\frac{p-1}{sp}}.
\end{equation}
Note that this is not the final choice of $c$ yet and it is subject to further smallness requirements in the course of the proof.

The pointwise estimate in \eqref{Eq:lower-bd} at $t_*=\bar{t}- \theta(\tfrac12 cR)^{sp}$ allows us to apply Lemma~\ref{Lm:DG:initial:1} with $\rho=\tfrac12cR$ and obtain that for some $\xi_o\in(0,\tfrac18)$, either
\begin{equation}\label{Eq:tail2:1}
	(\tfrac12c)^{\frac{sp}{p-1}} {\rm Tail}\big( \big(u-\boldsymbol \mu^{-}\big)_{-}; Q_o\big) > \xi_o   \boldsymbol\om,
\end{equation}
or
\begin{equation}\label{Eq:lower-bd:1}
 u\ge\boldsymbol \mu^-+\tfrac12\xi_o \boldsymbol\om
	\quad
	\mbox{a.e.~in $K_{\frac14cR}\times\big(t_*, t_*+\nu_o(\xi_o\boldsymbol\om)^{2-p}(\tfrac12cR)^{sp}\big]$.}
\end{equation}
The number $\xi_o$ is chosen to fulfill
\begin{equation}\label{Eq:xi}
\nu_o(\xi_o\boldsymbol\om)^{2-p}(\tfrac12cR)^{sp}\ge a(\tfrac14\boldsymbol\om)^{2-p}R^{sp},\quad\text{i.e.}\quad \xi_o=  \tfrac14\Big(\frac{\nu_o c^{sp}}{2^p a}\Big)^{\frac1{p-2}}.
\end{equation}
In this way, the estimate \eqref{Eq:lower-bd:1} can be claimed up to $t=0$ and yields the reduction of oscillation
\begin{equation}\label{Eq:osc1:0}
\essosc_{Q_{\frac14cR}(\theta)} u\le \big(1-\tfrac12\xi_o \big)\boldsymbol\om.
\end{equation}

Note that in the dependence of $\xi_o$, the constants $a$ and $c$ are still to be determined.
For the moment, let us suppose $c$ has been fixed. We then select $a$ to ensure \eqref{Eq:tail2:1} does not occur. In fact, since the tail is bounded by $\boldsymbol\gm\boldsymbol\om$ as before, we impose
\[
(\tfrac12c)^{\frac{sp}{p-1}}\boldsymbol\gm\boldsymbol\om\le \xi_o \boldsymbol\om\equiv \tfrac14\Big(\frac{\nu_o c^{sp}}{2^p a}\Big)^{\frac1{p-2}} \boldsymbol\om,
\]
where we have employed the selection of $\xi_o$ in \eqref{Eq:xi}. Consequently, the above display yields the relation of $a$ and $c$, that is,
\begin{equation}\label{Eq:a}
a=\frac{\nu_o}{\boldsymbol\gm}c^{\frac{sp}{p-1}}.
\end{equation}
Hence, by this choice, the estimate \eqref{Eq:tail2:1} does not occur and the reduction of oscillation \eqref{Eq:osc1:0} actually holds. 

Moreover, the assumption $a>2c^{sp}$ made at the beginning is verified, if we use \eqref{Eq:a} and further restrict $c$ by
\begin{equation}\label{Eq:c2:1.5}
a=\frac{\nu_o}{\boldsymbol\gm}c^{\frac{sp}{p-1}}>2c^{sp},\quad\text{i.e.}\quad c<\Big(\frac{\nu_o}{\boldsymbol\gm}\Big)^{sp\frac{p-1}{p-2}}.
\end{equation}
The number $a$ will be eventually fixed via \eqref{Eq:a}, once we determine $c$ in the end.

\subsection{The Second Alternative}
In this section, 
we work with $u$ as a sub-solution near its
supremum.
Suppose \eqref{Eq:1st-alt-meas} does not hold for any $\bar{t}\in\big(-a\theta R^{sp}+\theta(cR)^{sp},0\big]$. Due to \eqref{Eq:mu-pm-}, this can be rephrased as
\begin{equation*}
	\Big|\Big\{\boldsymbol\mu^+-u\ge\tfrac14 \boldsymbol\om\Big\}
	\cap 
	(0,\bar{t})+Q_{cR}(\theta)\Big|> \nu|Q_{cR}(\theta)|.
\end{equation*}
Based on this, it is not hard to find some $t_*\in\big[\bar{t}-\theta(cR)^{sp}, \bar{t}-\tfrac12\nu\theta(cR)^{sp}\big]$, such that
\begin{equation*}
	\Big|\Big\{\boldsymbol\mu^+-u(\cdot, t_*)\ge\tfrac14 \boldsymbol\om\Big\}
	\cap 
	 K_{cR} \Big|> \tfrac12\nu|K_{cR}|.
\end{equation*}
Indeed, if the above inequality were not to hold for any $s$ in the given 
interval, then 
\begin{align*}
	\Big|\Big\{\boldsymbol\mu^+-u\ge\tfrac14\boldsymbol\om\Big\}\cap 
	(0,\bar{t})+Q_{cR}(\theta)\Big|
	&=
	\int_{\bar{t}-\theta(cR)^p}^{\bar{t} -\frac12\nu\theta(cR)^{sp}}
	\Big|\Big\{\boldsymbol\mu^+-u(\cdot, s)\ge \tfrac{1}4\boldsymbol\om\Big\}
	\cap K_{cR}\Big|\,\ds\\
	&\phantom{=\,}
	+\int^{\bar{t}}_{\bar{t}-\frac12\nu\theta(cR)^{sp}}
	\Big|\Big\{\boldsymbol\mu^+-u(\cdot, s)\ge \tfrac{1}4\boldsymbol\om\Big\}
	\cap K_{cR}\Big|\,\ds\\
	&<\tfrac12\nu |K_{cR}|\theta(cR)^{sp}\big(1-\tfrac12\nu\big) +\tfrac12\nu\theta(cR)^{sp}
	|K_{cR} |\\
	&<\nu | Q_{cR}(\theta)|,
\end{align*}
which would yield a contradiction.

Starting from this measure theoretical information, we may apply Lemma~\ref{Lm:3:1} (with $\al=\tfrac12\nu$ and $\rho=cR$) to obtain $\dl$ and $\varep$ depending on the data and $\nu$, such that, for some $\xi_1\in(0,\frac14)$,
either 
\begin{equation}\label{Eq:tail3:0}
 c^{\frac{sp}{p-1}} {\rm Tail}\big( \big(\boldsymbol \mu^{+}-u\big)_{+}; Q_o\big) >\xi_1\boldsymbol \om,
\end{equation} 
 or
\begin{equation}\label{Eq:2nd-alt-meas:1}
	\Big|\Big\{
	 \boldsymbol \mu^{+}-u(\cdot, t) \ge \varep\xi_1\boldsymbol \om\Big\} \cap K_{cR} \Big|
	\ge\frac{\al}2 |K_{cR}|
	\>\mbox{ for all $t\in\big(t_*,t_*+\dl(\xi_1\boldsymbol \om)^{2-p}(cR)^{sp}\big]$.}
\end{equation}
The number $\xi_1$ is chosen to satisfy
\[
\dl(\xi_1\boldsymbol \om)^{2-p}(cR)^{sp}\ge\theta(cR)^{sp},\quad\text{i.e.}\quad\xi_1=\tfrac14\dl^{\frac1{p-2}}.
\]
In this way, the measure theoretical information \eqref{Eq:2nd-alt-meas:1} can be claimed up to the time level $\bar{t}$.
Whereas the constant $c$ is again chosen so small that \eqref{Eq:tail3:0} does not happen. A simple calculation as before gives
\begin{equation}\label{Eq:c2:1}
c\le\Big(\frac{\xi_1}{\boldsymbol \gm}\Big)^{\frac{p-1}{sp}}.
\end{equation}
Consequently, the measure theoretical information \eqref{Eq:2nd-alt-meas:1} yields
\begin{equation}\label{Eq:2nd-alt-meas:2}
	\Big|\Big\{
	 \boldsymbol \mu^{+}-u(\cdot, t) \ge \varep\xi_1\boldsymbol \om\Big\} \cap K_{cR} \Big|
	\ge\frac{\al}2 |K_{cR}|
	\>\mbox{ for all $t\in\big(-a\theta R^{sp}+\theta (cR)^{sp},0\big]$,}
\end{equation}
thanks to the arbitrariness of $\bar{t}$. 

Given \eqref{Eq:2nd-alt-meas:2}, we want to apply Lemma~\ref{Lm:3:2} with $\dl=1$, $\xi=\varep\xi_1$ and $\rho=cR$ next. 
To this end, we first let $\nu$ be fixed in Lemma~\ref{Lm:DG:1} (with $\dl=1$ ) and choose $\sig\in(0,\tfrac12)$ to satisfy
\[\boldsymbol\gm \frac{\sig^{p-1}}{\al}\le\nu.\]
Then we use \eqref{Eq:a} and restrict $c$ further to satisfy
\begin{equation}\label{Eq:c2:2}
a\theta R^{sp}-\theta (cR)^{sp}\ge(\sig\varep\xi_1\boldsymbol\om)^{2-p}(cR)^{sp},\quad\text{i.e.}\quad c\le\Big(\frac{\nu_o}{\boldsymbol\gm}\Big)^{\frac{p-1}{sp(p-2)}}(\sig\varep\xi_1)^{\frac{p-1}{sp}}.
\end{equation}
In this way, the measure theoretical information \eqref{Eq:2nd-alt-meas:2} gives that
	\begin{equation*}
	\Big|\Big\{
		\boldsymbol \mu^{+}-u(\cdot, t)\ge \varep\xi_1\boldsymbol \om
		\Big\}\cap K_{cR}\Big|
		\ge\al \big|K_{cR}\big|\quad\mbox{ for all $t\in\big( -(\sig\varep\xi_1\boldsymbol \om)^{2-p}(cR)^p, 0\big]$,}
	\end{equation*}
which allows us to implement Lemma~\ref{Lm:3:2}.
Namely, there exists
 $\boldsymbol \gm>0$ depending only on the data,  such that 
  either 
 \begin{equation}\label{Eq:tail3:1}
  c^{\frac{sp}{p-1}} {\rm Tail}\big( \big(\boldsymbol \mu^{+}-u\big)_{+}; Q_o\big) >\sig\varep\xi_1\boldsymbol\om
 \end{equation} 
 or
\begin{equation*}
	\Big|\Big\{
	\boldsymbol \mu^{+}-u \le \sig \varep\xi_1\boldsymbol \om \Big\}\cap Q_{cR}(\theta)\Big|
	\le \nu \big|Q_{cR}(\widetilde\theta)|,\quad\text{where}\>\widetilde\theta=(\sig\varep\xi_1\boldsymbol \om)^{2-p}.
\end{equation*}
By Lemma~\ref{Lm:DG:1} (with $\dl=1$), the last display yields
\[
\boldsymbol \mu^{+}-u \ge \tfrac12\sig \varep\xi_1\boldsymbol \om\quad\text{a.e. in}\>Q_{\frac12cR}(\widetilde\theta),
\]
which in turn gives the reduction of oscillation
\begin{equation}\label{Eq:osc1:1}
\essosc_{Q_{\frac12cR}(\widetilde\theta)} u\le \big(1- \tfrac12\sig \varep\xi_1\big)\boldsymbol \om.
\end{equation}

Once again, we may restrict the choice of $c$ to ensure that \eqref{Eq:tail3:1} does not happen. By a similar calculation as before, this amounts to requiring
\begin{equation}\label{Eq:c2:3}
c\le\Big(\frac{\sig\varep\xi_1}{\boldsymbol \gm}\Big)^{\frac{p-1}{sp}}.
\end{equation}

Combining \eqref{Eq:osc1:0} and \eqref{Eq:osc1:1}, we arrive at
\begin{equation}\label{Eq:osc1:2}
\essosc_{Q_{\frac14cR}( \theta)} u\le \big(1- \eta\big)\boldsymbol \om:=\boldsymbol \om_1,
\end{equation}
where
\[
\eta=\min\big\{\tfrac12\xi_o, \tfrac12\sig \varep\xi_1\big\}.
\]

Now, set $\theta_1=(\tfrac14\boldsymbol \om_1)^{2-p}$ and $R_1=\tfrac14 c R$. To prepare the induction, we need to verify the set inclusion
\[
Q_{R_1}(a\theta_1)\subset Q_{\frac14cR}( \theta),\quad\text{i.e.}\quad a\le \big[\tfrac14(1-\eta)\big]^{p-2}.
\]
This, by the choice of $a$ in \eqref{Eq:a}, amounts to further requiring the smallness of $c$.

As a result of this inclusion and  \eqref{Eq:osc1:2}, we obtain
\begin{equation*}
\essosc_{Q_{R_1}( a\theta_1)} u\le \boldsymbol \om_1,
\end{equation*}
which takes the place of \eqref{Eq:start-cylinder:1} in the next stage. Note that this oscillation estimate also takes into account the reverse case of \eqref{Eq:mu-pm-}.

\subsection{The Induction}
Now we may proceed by induction. 
 Suppose up to $i=0,1,\cdots, j$, we have built
\begin{equation*}
\left\{
	\begin{array}{c}
	\dsty R_o=R,\quad R_i=\tfrac14c R_{i-1},
	\quad 
	\boldsymbol\om_i=(1-\eta)\boldsymbol\om_{i-1},\\[5pt]
	\theta_i=(\tfrac14\boldsymbol\om_i)^{2-p},
	\quad 
	Q_i=Q_{R_i}(a\theta_i), 
	\\[5pt]
	\dsty\boldsymbol\mu_i^+=\essup_{Q_i}u,
	\quad
	\boldsymbol\mu_i^-=\essinf_{Q_i}u,
	\quad
	\essosc_{Q_i}u\le\boldsymbol\om_i.
	\end{array}
\right.
\end{equation*}
The induction argument will show that the above oscillation estimate continues to hold for the $(j+1)$-th step.

Like in the proof for $1<p\le2$, we can repeat all the previous arguments, which now are adapted with  $\boldsymbol\mu^{\pm}_j$, $\boldsymbol\om_j$, $R_j$, $\theta_j$, $Q_j$, etc. In the end, we have a reduction of oscillation parallel with \eqref{Eq:osc1:2}, that is,
\begin{equation}\label{Eq:osc1:3}
\essosc_{Q_{\frac14cR_j}( \theta_j)} u\le \big(1- \eta\big)\boldsymbol \om_j:=\boldsymbol \om_{j+1}.
\end{equation}
Now, setting $\theta_{j+1}=(\tfrac14\boldsymbol \om_{j+1})^{2-p}$ and $R_{j+1}=\tfrac14 c R_j$, it is straightforward to verify the set inclusion
\[
Q_{R_{j+1}}(a\theta_{j+1})\subset Q_{\frac14cR_j}( \theta_j), 
\]
which, by \eqref{Eq:osc1:3}, implies
\[
\essosc_{Q_{R_{j+1}}( a\theta_{j+1})} u \le \boldsymbol \om_{j+1}.
\]

 A key step lies in determining $c$ as done in \eqref{Eq:c2:0}, \eqref{Eq:c2:1}, \eqref{Eq:c2:1.5}, \eqref{Eq:c2:2} and \eqref{Eq:c2:3}, such that the alternative involving the tail, along the course of the arguments, does not really occur and hence \eqref{Eq:osc1:3} can be reached.
This hinges upon the following estimate of the tail:
\[
 {\rm Tail}\big(\big(u - \boldsymbol \mu_j^{\pm}\big)_{\pm}; Q_j\big)  \le \boldsymbol\gm\boldsymbol\om_j.
\]
The computations leading to the above tail estimate can be performed as those leading to \eqref{Eq:tail-est}. We omit the details to avoid repetition. After the number $c$ is determined independent of $j$, the number $a$ is finally chosen via the relation \eqref{Eq:a}. Hence the induction is completed and the derivation of a H\"older modulus of continuity follows. 

\appendix
\section{Fractional Sobolev Inequalities}\label{A:1}
Note that the definition of the space $W^{s,p}$ in Section~\ref{S:notion} is also valid for $p=1$, though we will not use it. An elementary proof of the following result can be retrieved from \cite[Theorem~6.5]{Hitchhiker}. 
\begin{proposition}\label{Sobolev-0}
Let $s\in(0,1)$ and $p\ge1$ satisfying $sp<N$. For any measurable and compactly supported function $u:\rn\to\rr$, there holds
\[
\Big(\int_{\rn} |u|^{\frac{Np}{N-sp}}\,\dx\Big)^{\frac{N-sp}{N}}\le C \int_{\rn}\int_{\rn}\frac{\big|u(x) - u (y)\big|^p}{|x-y|^{N+sp}}\,\dx\dy
\]
for some positive constant $C=C(s,p,N)$.
\end{proposition}

 It is our intention to circumvent a more advanced theory of function spaces, extension domains, etc. and keep this section as elementary as possible. 

The following local version is a direct consequence of Proposition~\ref{Sobolev-0}.
Here and in the sequel, we omit the reference to the center of a ball $K_R$.
\begin{proposition}\label{Sobolev-1}
Let $s\in(0,1)$ and $p\ge1$. For any function 
$u\in W^{s,p}(K_R)$ that is compactly supported in $K_{(1-d)R}$ with some $d\in(0,1)$, there holds
\[
\Big(\bint_{K_R} |u|^{\kappa p}\,\dx\Big)^{\frac1\kappa}\le C R^{sp} \int_{K_R}\bint_{K_R}\frac{\big|u(x) - u (y)\big|^p}{|x-y|^{N+sp}}\,\dx\dy+\frac{C}{d^{N+sp}}\bint_{K_R}|u(x)|^p\,\dx
\]
for some positive constant $C=C(\kappa, s, p, N)$,  
where
\begin{equation*}
\kappa \in \left\{
\begin{array}{cl}
\Big[1,\frac{N}{N-sp}\Big],\quad  &sp<N,\\[5pt]
[1,\infty),\quad &sp\ge N. 
\end{array}
\right.
\end{equation*}
\end{proposition}
\begin{proof}
It suffices to show the inequality for $R=1$, thanks to its scaling invariance in $R$.
Since $u$ is compactly supported in $K_1$,
we may view it as a function in $W^{s,p}(\rn)$ upon zero extension. This fact is verified by a simple calculation (similar to what follows). First consider the case $sp<N$.  According to Proposition~\ref{Sobolev-0}, we estimate
\begin{equation}\label{Eq:sp<N}
\begin{aligned}
\Big(\int_{\rn} &|u|^{\frac{Np}{N-sp}}\,\dx\Big)^{\frac{N-sp}{N}} \\
&\le C \int_{\rn}\int_{\rn}\frac{\big|u(x) - u (y)\big|^p}{|x-y|^{N+sp}}\,\dx\dy\\
&= C \int_{K_1}\int_{K_1}\frac{\big|u(x) - u (y)\big|^p}{|x-y|^{N+sp}}\,\dx\dy 
+ 2C\int_{K_1}\int_{\rn\setminus K_1}\frac{|u(y) |^p}{|x-y|^{N+sp}}\,\dx\dy\\
&\le C \int_{K_1}\int_{K_1}\frac{\big|u(x) - u (y)\big|^p}{|x-y|^{N+sp}}\,\dx\dy 
+ \frac{C}{d^{N+sp}}\int_{K_1}  |u(y) |^p \,\dy.
\end{aligned}
\end{equation}
In the last line, we used the fact that $y\in K_{1-d}$ (recalling $\supp u$) and $x\in\rn\setminus K_1$.
Consequently, the desired inequality for any $\kappa$ in the interval follows from an application of H\"older's inequality. Note also the constant $C$ is actually independent of $\kappa$ in this case.

Next, we consider the case $sp\ge N$. A first observation is that $W^{\bar{s},p}(K_1)\subset W^{s,p}(K_1)$ for any $0<\bar{s}\le s<1$. Quantitatively, we estimate
\begin{equation*}
\begin{aligned}
&\int_{K_1}\int_{K_1}\frac{\big|u(x) - u (y)\big|^p}{|x-y|^{N+\bar{s}p}}\,\dx\dy\\
&\quad= \int_{K_1}\int_{K_1\cap\{|x-y|<1\}}\frac{\big|u(x) - u (y)\big|^p}{|x-y|^{N+\bar{s}p}}\,\dx\dy + \int_{K_1}\int_{K_1\cap\{|x-y|\ge1\}}\frac{\big|u(x) - u (y)\big|^p}{|x-y|^{N+\bar{s}p}}\,\dx\dy\\
&\quad\le \int_{K_1}\int_{K_1}\frac{\big|u(x) - u (y)\big|^p}{|x-y|^{N+sp}}\,\dx\dy + C(\bar{s},p,N)\int_{K_1}|u(x)|^p\,\dx.
\end{aligned}
\end{equation*}
Now we select $\bar{s}=\frac{(\kappa-1)}{\kappa p}N$, which verifies $\bar{s}p<N$, and apply \eqref{Eq:sp<N} in $W^{\bar{s},p}(K_1)$, together with the above observation, to conclude.
\end{proof}

Based on Proposition~\ref{Sobolev-1}, the following parabolic imbedding is in order. It plays an essential role in proving DeGiorgi type lemmas.
\begin{proposition}\label{Sobolev-2}
Let $s\in(0,1)$, $p\ge1$ and 
\begin{equation*}
\kappa_* := \left\{
\begin{array}{cl}
 \frac{N}{N-sp} ,\quad  &sp<N,\\[5pt]
2,\quad &sp\ge N.
\end{array}
\right.
\end{equation*}
For any function
\[
u\in L^{p}\big(t_1,t_2; W^{s,p}(K_R)\big)\cap L^{\infty}\big(t_1,t_2; L^{2}(K_R)\big),
\]
which is compactly supported in $K_{(1-d)R}$ for some $d\in(0,1)$ and  for a.e. $t\in(t_1,t_2)$,
there holds 
\begin{equation*}
\begin{aligned}
\int_{t_1}^{t_2}&\int_{K_R}|u(x,t)|^{\kappa p}\,\dx\dt\\
&\le C \Big(R^{sp}\int_{t_1}^{t_2}\int_{K_R} \int_{K_R}\frac{\big|u(x,t) - u (y,t)\big|^p}{|x-y|^{N+sp}}\,\dx\dy\dt+ \frac1{d^{N+sp}}\int_{t_1}^{t_2}\int_{K_R} |u(x,t)|^p\,\dx\dt\Big)\\
&\quad\times\Big(\essup_{t_1<t<t_2}\bint_{K_R} |u(x,t)|^2\,\dx\Big)^{\frac{\kappa_* -1}{\kappa_*}},
\end{aligned}
\end{equation*}
for some positive constant $C=C(s,p,N)$,
where
\[
\kappa:=1+\frac{2(\kappa_* -1)}{p\kappa_*}.
\]
\end{proposition}
\begin{proof} 
The inequality is scaling invariant in $R$. Hence, it suffices to show it for $R=1$.
An application of H\"older's inequality, followed by Proposition~\ref{Sobolev-1} and the choice of $\kappa$, yields that
\begin{align*}
\int_{t_1}^{t_2}&\int_{K_1}|u |^{\kappa p}\,\dx\dt=\int_{t_1}^{t_2}\int_{K_1}|u |^{  p}|u |^{(\kappa-1) p}\,\dx\dt\\
&\le\int_{t_1}^{t_2}\Big(\int_{K_1}|u |^{\kappa_* p}\,\dx\Big)^{\frac1{\kappa_*}}
\Big(\int_{K_1}|u |^{p\frac{\kappa_* (\kappa -1)}{\kappa_*- 1}}\,\dx\Big)^{\frac{1-\kappa_*}{\kappa_*}}\dt \\
&\le C \Big( \int_{t_1}^{t_2}\int_{K_1} \int_{K_1}\frac{\big|u(x,t) - u (y,t)\big|^p}{|x-y|^{N+sp}}\,\dx\dy\dt+ \frac1{d^{N+sp}}\int_{t_1}^{t_2}\int_{K_1} |u(x,t)|^p\,\dx\dt\Big)\\
&\quad\times\Big(\essup_{t_1<t<t_2}\int_{K_1} |u(x,t)|^2\,\dx\Big)^{\frac{\kappa_* -1}{\kappa_*}}.
\end{align*}
This finishes the proof.
\end{proof}

\section{Time Mollification}\label{A:2}
The time derivative of a weak solution  in general  does not exist in the Sobolev sense. On the other hand, it is desirable to use the solution in a testing function when we derive the energy estimate in Proposition~\ref{Prop:2:1}. To overcome this well-organized difficulty, we find the following mollifier quite convenient. 
Namely, we introduce for any $v\in L^1(E_T)$,
\begin{equation*}
	\llbracket v \rrbracket_h(x,t)
	:= 
	\tfrac 1h \int_0^t \mathrm e^{\frac{\tau-t}h} v(x,\tau) \, \d \tau,\quad
	\llbracket v \rrbracket_{\bar{h}}(x,t)
	:= 
	\tfrac 1h \int_t^T \mathrm e^{\frac{t-\tau}h} v(x,\tau) \, \d \tau.
\end{equation*}
Properties of this mollification can for instance be found in \cite[Lemma~2.2]{KL}.
In particular, by direct differentiation, we obtain the following identities:
\begin{equation}\label{dt-mol}
	 \partial_t \llbracket v \rrbracket_h
	 =
	 \tfrac{1}{h} \big(v-\llbracket v \rrbracket_h\big), \quad
	  \partial_t \llbracket v \rrbracket_{\bar{h}}
	 =
	 \tfrac{1}{h} \big(\llbracket v \rrbracket_{\bar{h}}- v\big).
\end{equation}
They are paired by the identity
\[
\int_0^T \partial_t \llbracket v \rrbracket_h\, \phi\,\dt =- \int_0^T \llbracket v \rrbracket_{\bar{h}} \, \partial_t\phi\,\dt\quad\text{for all}\>\phi\in C_o^1\big((0,T)\big).
\]
Another fact, we will reply on is that, if $v\in L^q(E_T)$ for some $q\ge1$, then $\llbracket v \rrbracket_h,\, \llbracket v \rrbracket_{\bar{h}}\in L^q(E_T)$, and moreover, as $h\to0$ we have
\begin{equation}\label{Lp-conv}
\llbracket v \rrbracket_h,\, \llbracket v \rrbracket_{\bar{h}}\to v\quad\text{in}\> L^q(E_T).
\end{equation}

Now we take on a rigorous justification of formal calculations in Proposition~\ref{Prop:2:1}.
For ease of notation, we denote
\[
\mathcal{A}\equiv K(x,y,t)\big|u(x,t) - u(y,t)\big|^{p-2} \big(u(x,t) - u(y,t)\big).
\]
Let the cutoff function $\z$ be as in Proposition~\ref{Prop:2:1} and introduce $\psi_\varep(t),\,\varep>0$ to be a
Lipschitz approximation of $\chi_{\{t_o-S<t<t_o\}}$. Namely, $\psi_\varep=1$ in $(t_o-S+\varep, t_o-\varep)$, $\psi_\varep=0$ outside $(t_o-S,t_o)$, while linearly interpolated otherwise. In the weak formulation \eqref{Eq:1:4p}, we take the testing function
\[
\vp_h=\big(\llbracket u \rrbracket_{\bar{h}}-k\big)_+\z^p\psi_{\varep};
\]
we obtain
\begin{equation}\label{Eq:A:1}
-\int_0^T\int_E u \pl_t \vp_h\,\dx\dt 
+\int_0^T\int_{\rn}\int_{\rn} \mathcal{A} \cdot \big(\vp_h(x,t)-\vp_h(y,t)\big)\,\dy\dx\dt\le0.
\end{equation}

The first term in \eqref{Eq:A:1} can be written as, recalling \eqref{dt-mol},
\begin{align*}
-\int_0^T\int_E u \pl_t \vp_h\,\dx\dt &= -\int_0^T\int_E \llbracket u \rrbracket_{\bar{h}} \pl_t \vp_h\,\dx\dt
+\int_0^T\int_E \big(\llbracket u \rrbracket_{\bar{h}} - u\big) \pl_t \vp_h\,\dx\dt\\
&=\tfrac12\int_0^T\int_E  \pl_t\big(\llbracket u \rrbracket_{\bar{h}} -k\big)_+^2 \z^p\psi_{\varep}\,\dx\dt\\
&\quad +\tfrac1h\int_0^T\int_E \big(\llbracket u \rrbracket_{\bar{h}} - u\big)^2 \chi_{\{\llbracket u \rrbracket_{\bar{h}}>k\}}\z^p\psi_{\varep}\,\dx\dt\\
&\quad + \int_0^T\int_E \big(\llbracket u \rrbracket_{\bar{h}} - u\big)\big(\llbracket u \rrbracket_{\bar{h}} -k\big)_+ \pl_t\big(\z^p\psi_{\varep}\big)\,\dx\dt.
\end{align*}
On the right-hand side of the last display, the second term is discarded due to its non-negative contribution, whereas the third term converges to $0$ as $h\to0$ owing to \eqref{Lp-conv}.
For the first term, we integrate by parts in time, let $h\to0$ and then let $\varep\to0$, to obtain
\[
\tfrac12\int_{K_R}\big(u-k\big)_+^2\z^p\,\dx\bigg|_{t_o-S}^{t_o} - \tfrac12\iint_{Q(R,S)}\big(u-k\big)_+^2\pl_t\z^p\,\dx\dt.
\]
These are the last terms on the left/right-hand side of the energy estimate in Proposition~\ref{Prop:2:1}.

The next goal is to show that the second integral in \eqref{Eq:A:1} converges to
\[
\int_0^T\int_{\rn}\int_{\rn} \mathcal{A} \cdot \big(\vp(x,t)-\vp(y,t)\big)\,\dy\dx\dt,
\]
if we send $h\to0$, where
\[
\vp=\big( u -k\big)_+\z^p\psi_{\varep}.
\]
Once this is shown, we will justify the formal calculations in the proof of Proposition~\ref{Prop:2:1}.

For this purpose, we take the difference between the two integrals and split the obtained integral into two parts, that is,
\begin{align*}
&I_1:= \int_0^T\int_{K_{2R}}\int_{K_{2R}}\mathcal{A} \cdot \Big(\big[\vp_h(x,t)-\vp_h(y,t)\big] - \big[\vp(x,t)-\vp(y,t)\big]\Big)\,\dy\dx\dt,\\
&I_2:=2\int_0^T\int_{K_{2R}}\int_{\rn\setminus K_{2R}}\mathcal{A} \cdot \big[\vp_h(x,t)- \vp(x,t) \big]\,\dy\dx\dt.
\end{align*}
Note that we used $\vp_h(y,t)=\vp(y,t)=0$ for $y\in \rn\setminus K_{2R}$ in defining $I_2$.
We then show that $I_1$ and $I_2$ converge to $0$ after sending $h\to0$. 

For $I_1$, we first claim that
\begin{equation}\label{Eq:A:uniform-bd}
\int_0^T\int_{K_{2R}}\int_{K_{2R}} \frac{ \big|\vp_h(x,t)-\vp_h(y,t)\big|^p}{|x-y|^{N+sp}}\,\dy\dx\dt \le C   \|u\|^p_{L^p(0,T; W^{s,p}(K_{2R}))},
\end{equation}
for some constant $C=C(s,p,N, R, \|D\z\|_{\infty} )$.
This uniform bound allows us to extract a weakly convergent subsequence (still labeled by $h$). Since, up to a subsequence, $\vp_h\to\vp$ a.e. in $E_T$, we obtain that, up to a subsequence,
\[
 \frac{ \vp_h(x,t)-\vp_h(y,t)}{|x-y|^{\frac{N}{p}+s}} \rightharpoonup  \frac{ \vp(x,t)-\vp(y,t)}{|x-y|^{\frac{N}{p}+s}}\quad\text{weakly in}\>L^p\big((0,T)\times K_{2R}\times K_{2R}\big).
\]
Consequently, we have $I_1\to0$ in view of this weak convergence and the fact that
\[
\mathcal{A}\cdot |x-y|^{\frac{N}{p}+s}\in L^{\frac{p}{p-1}}\big((0,T)\times K_{2R}\times K_{2R}\big).
\]

To prove \eqref{Eq:A:uniform-bd},  the numerator of the integrand needs to be controlled. Using the triangle inequality and the fact that $\z$ and $\psi_\varep$ are bounded by $1$, we compute
\[
\big|\vp_h(x,t)-\vp_h(y,t)\big|^p\le c \big|\llbracket u \rrbracket_{\bar{h}}(x,t) - \llbracket u \rrbracket_{\bar{h}}(y,t)\big|^p  + c\big| \llbracket u \rrbracket_{\bar{h}}(y,t)\big|^p \big|\z(x,t)-\z(y,t)\big|^p
\]
for some proper $c=c(p)$.
Next, we estimate the first term on the right-hand side of the last display by H\"older's inequality:
\begin{align*}
\big|\llbracket u \rrbracket_{\bar{h}}(x,t) - \llbracket u \rrbracket_{\bar{h}}(y,t)\big|^p&\le \bigg(\tfrac 1h \int_t^T \mathrm e^{\frac{t-\tau}h}  \, \d \tau\bigg)^{p-1}\bigg(\tfrac 1h \int_t^T \mathrm e^{\frac{t-\tau}h}  \big| u(x,\tau) - u(y,\tau)\big|^p \, \d \tau\bigg)\\
&\le \tfrac 1h \int_t^T \mathrm e^{\frac{t-\tau}h}  \big| u(x,\tau) - u(y,\tau)\big|^p \, \d \tau.
\end{align*}
A time integration of the above estimate, followed by Fubini's theorem, gives that
\begin{align*}
\int_0^T \big|\llbracket u \rrbracket_{\bar{h}}(x,t) - \llbracket u \rrbracket_{\bar{h}}(y,t)\big|^p\,\dt&\le \int_0^T\big(1-\mathrm e^{-\frac{t}{h}}\big) \big| u(x,t) - u(y,t)\big|^p\,\dt\\
&\le \int_0^T  \big| u(x,t) - u(y,t)\big|^p\,\dt.
\end{align*}
Thus we deduce that
\begin{equation}\label{Eq:A:4}
\int_0^T\int_{K_{2R}}\int_{K_{2R}} \frac{\big|\llbracket u \rrbracket_{\bar{h}}(x,t) - \llbracket u \rrbracket_{\bar{h}}(y,t)\big|^p}{|x-y|^{N+sp}}\,\dy\dx\dt \le 
  \|u\|^p_{L^p(0,T; W^{s,p}(K_{2R}))}.
\end{equation}
Analogously, for the second term, we may first estimate 
\begin{align*}
\int_0^T \big|  \llbracket u \rrbracket_{\bar{h}}(y,t)\big|^p\,\dt \le \int_0^T  \big| u(y,t)\big|^p\,\dt,
\end{align*}
and then,
\begin{equation}\label{Eq:A:5}
\begin{aligned}
\int_0^T&\int_{K_{2R}}\int_{K_{2R}} \frac{\big| \llbracket u \rrbracket_{\bar{h}}(y,t)\big|^p \big|\z(x,t)-\z(y,t)\big|^p}{|x-y|^{N+sp}}\,\dy\dx\dt \\
&\le\|D\z\|^p_{\infty} \int_0^T\int_{K_{2R}}\int_{K_{2R}} \frac{\big|\llbracket u \rrbracket_{\bar{h}}(y,t)\big|^p |x-y|^p}{|x-y|^{N+sp}}\,\dy\dx\dt \\
&\le \|D\z\|^p_{\infty} \int_0^T\int_{K_{2R}}\int_{K_{2R}} \frac{\big|u(y,t)\big|^p  }{|x-y|^{N+(s-1)p}}\,\dy\dx\dt\\
&\le \boldsymbol\gm  \|D\z\|^p_{\infty}  R^{(1-s)p} \|u\|^p_{L^p(0,T; L^{p}(K_{2R}))}.
\end{aligned}
\end{equation}
Combining \eqref{Eq:A:4} and \eqref{Eq:A:5}, we have shown \eqref{Eq:A:uniform-bd}.

To deal with $I_2$, we first notice  that $|x-y|\ge\tfrac12 |y|$ for $y\in \rn\setminus K_{2R}$ and $x\in  K_{R}$. Hence, we estimate
\begin{align*}
| I_2 | &\le 2 \int_0^T\int_{K_{2R}}\int_{\rn\setminus K_{2R}} \Big| \mathcal{A} \cdot \big(\vp_h(x,t)-\vp(x,t)\big) \Big| \,\dy\dx\dt\\
& \le 2C_1 \int_0^T\int_{K_{2R}}\int_{\rn\setminus K_{2R}}  \frac{|u(x,t) - u(y,t)|^{p-1}}{|x-y|^{N+sp}} \big| \vp_h(x,t)-\vp(x,t)  \big| \,\dy\dx\dt\\
& \le \boldsymbol\gm \int_0^T\int_{K_{2R}}  \big|  \vp_h(x,t)-\vp(x,t)  \big| \,\dx\dt\int_{\rn\setminus K_{2R}}  \frac{|u(x,t)|^{p-1} + |u(y,t)|^{p-1}}{|y|^{N+sp}}\,\dy\\
& \le \frac{\boldsymbol\gm}{R^{sp}}  \int_0^T\int_{K_{R}}  \big|  \vp_h(x,t)-\vp(x,t)  \big| \Big( |u(x,t)|^{p-1} + \big[{\rm Tail}\big(u; Q(R,S)\big)\big]^{p-1}\Big)\,\dx\dt\\
& \le \frac{\boldsymbol\gm}{R^{sp}} \bigg[\iint_{Q(R,S)}  \big|  \vp_h -\vp   \big|^p\,\dx\dt\bigg]^{\frac1p}
 \bigg[\iint_{Q(R,S)} \Big(|u |^{p } + \big[{\rm Tail}\big(u; Q(R,S)\big)\big]^{p}\Big)\,\dx\dt\bigg]^{\frac{p-1}{p}}.
\end{align*}
Now we easily see that $I_2\to0$ as $h\to0$ due to \eqref{Lp-conv}.



\end{document}